\documentclass[11pt]{amsart}
\usepackage[a4paper, total={6in, 8in}]{geometry}
\usepackage{enumerate}
\usepackage{amsmath}
\usepackage{amssymb,latexsym}
\usepackage{amsthm}
\usepackage{color}
\usepackage{cancel}
\usepackage{graphicx}
\usepackage{cite}
\usepackage{changes}
\usepackage{hyperref}
\usepackage{comment}
\usepackage{amscd}
\usepackage{mathtools}

\DeclareMathOperator{\C}{\mathcal{C}}

\usepackage[foot]{amsaddr}
\DeclareMathOperator{\GammaL}{\Gamma\mathrm{L}}
\DeclareMathOperator{\Aut}{Aut}

\newtheorem{theorem}{Theorem}[section]

\newtheorem{lemma}[theorem]{Lemma}
\newtheorem{corollary}[theorem]{Corollary}

\newtheorem{proposition}[theorem]{Proposition}
\newtheorem{construction}[theorem]{Construction}

\newtheorem{remark}[theorem]{Remark}

\newcommand{\fqn}{\mathbb{F}_{q^n}}

\newcommand{\F}{{\mathbb F}}

\newcommand{\fq}{{\mathbb F}_{q}}

\newcommand{\PG}{\mathrm{PG}}
\newcommand{\PGL}{\mathrm{PGL}}
\newcommand{\N}{\mathrm{N}}

\title{Clubs and their applications}
\date{}

\author[V. Napolitano]{Vito Napolitano}
\author[O. Polverino]{Olga Polverino}
\author[P. Santonastaso]{Paolo Santonastaso}
\author[F. Zullo]{Ferdinando Zullo}
\address{Vito Napolitano, Olga Polverino, Paolo Santonastaso and Ferdinando Zullo, \textnormal{Dipartimento di Matematica e Fisica, Universit\`a degli Studi della Campania ``Luigi Vanvitelli'', Viale Lincoln, 5, I--\,81100 Caserta, Italy}}
\email{\{vito.napolitano,olga.polverino,paolo.santonastaso,ferdinando.zullo\}@unicampania.it}

\begin{document}

\maketitle

\begin{abstract}
Clubs of rank $k$ are well-celebrated objects in finite geometries introduced by Fancsali and Sziklai in 2006. After the connection with a special type of arcs known as KM-arcs, they renewed their interest.
This paper aims to study clubs of rank $n$ in $\PG(1,q^n)$.
We provide a classification result for $(n-2)$-clubs of rank $n$, we analyze the $\GammaL(2,q^n)$-equivalence of the known subspaces defining clubs, for some of them the problem is then translated in determining whether or not certain scattered spaces are equivalent. Then we find a polynomial description of the known families of clubs via some linearized polynomials. Then we apply our results to the theory of blocking sets, KM-arcs, polynomials and rank metric codes, obtaining new constructions and classification results.
\end{abstract}



\section{Introduction}

Linear sets have been shown to be a powerful tool in several classification results and constructions of objects in finite geometry and coding theory, such as semifields, blocking sets, translation ovoids, KM-arcs, rank metric codes, etc. For more details on their applications we refer to \cite{LavVdV,Polverino,PZ}. The main topic of this paper regards a special class of linear sets on the projective line called \emph{clubs}.
Let $\Lambda=\PG(V,\fqn)=\PG(1,q^n)$, where $V$ is an $\fqn$-vector space of dimension $2$. 
Let $U$ be an $\fq$-subspace of $V$, then the set
\[ L_U=\{\langle {u} \rangle_{\fqn} \colon {u}\in U\setminus\{{0}\}\} \]
is said to be an $\fq$-\emph{linear set} of rank $\dim_{\fq}(U)$ of $\Lambda$. Another important notion associated with linear sets is those of \emph{weight of a point} $P$ which, roughly speaking, represents a measure of how much of the linear set is contained in the point. We say that a linear set is \emph{scattered} if all of its points have weight one, that is its points share with the linear set the minimum possible. 
The term scattered was introduced for the first time in \cite{BL2000} by Blokhuis and Lavrauw, in which the authors studied this notion in a more general framework. Scattered linear sets recently gained more attention due to their applications in the theory of Maximum Rank Distance codes and more generally to rank metric codes; see \cite{PZ}.
A quite close to the scattered linear sets family is the one of \emph{clubs}. An $i$-\emph{club} of rank $k$ is an $\fq$-linear set $L_U$ of rank $k$ for which all but one of the points have weight one, the remaining one has weight $i$. In the case in which the rank is $n$ we will simply say that $L_U$ is an $i$-\emph{club}.
They were originally introduced by in Fancsali and Sziklai in \cite{FanSzi1} (see also \cite{FanSzi2}) when studying maximal partial $2$-spreads in $\PG(8,q)$. These linear sets renewed their interest when De Boeck and Van de Voorde in \cite{DeBoeckVdV2016} characterized the translation KM-arcs exactly as those that can be described by $i$-clubs in even characteristic and such a connection enabled them to give new constructions and classifications of KM-arcs. Moreover, $i$-clubs also define linear blocking sets of R\'edei type once we look to the $i$-club as the set of determined directions of an affine pointset. Apart for their own interesting combinatorial properties, these pointsets turn out to be also very interesting since they also define Hamming metric codes with \emph{few weights}; see \cite{NPSZ2022codes}.
Furthermore, $i$-clubs have a very natural algebraic description via linearized polynomials, which have been recently investigated in \cite{BMZZ}, under the name of $1$-\emph{fat polynomials}.
In particular, in \cite{BMZZ} it has been proved that \emph{exceptional} $1$-fat polynomials do not exist, which means that the nature of these polynomial heavily rely on the extension field in which we are looking for this kind of polynomials.

\vspace{0.5cm}

In this paper we start by providing a classification for $(n-2)$-clubs in $\PG(1,q^n)$, by first studying a class of linear sets which, up to equivalence, contains all the $(n-2)$-clubs and then using an algebraic result which extends the linear analogue of Vosper's theorem under certain assumptions. Then we study the $\GammaL(2,q^n)$-equivalence of the known subspaces defining clubs (not necessarily $(n-2)$-clubs), in which the scattered spaces play an important role. Then we find linearized polynomials defining the known families of clubs, in particular for the $(n-2)$-clubs we find all of them (up to equivalence). Finally we apply our results to the theory of blocking sets, KM-arcs and rank metric codes, in such a way that we obtain constructions and classification results. In particular, our results imply, as a special case, the classification results of translation KM-arcs of type $q/4$ provided in \cite{DeBoeckVdV2016}, but with an explicit description of the associated clubs.

\vspace{0.5cm}

The paper is organized as follows: in Section \ref{sec:prel} we will briefly describe some of the objects and related results we will use. More precisely, we recall linearized polynomials, linear sets and clubs, dual basis and some other algebraic results. In Section \ref{sec:classn-2}, we first detect a family of linear sets of rank $h+2$ containing a point of weight $h$, then we show that, up to equivalence, all the $(n-2)$-clubs belong to such a family. Therefore, we characterize the linear sets belonging to such a family which turn out to be clubs via an algebraic description which allowed us to use a generalization (under certain assumption) of the linear analogue of the Vosper's theorem. Section \ref{sec:equivalence} is devoted to the study of the $\GammaL(2,q^n)$-equivalence of the known subspaces whose associated linear set defines an $i$-club. Our analysis shows that there is a wide variety of examples. Indeed, for the constructions shown by G\'acs and Weiner in \cite{GW2003} we showed that the number of inequivalent subspaces is related to the number of scattered subspaces in a smaller extension. We also detect some inequivalence subspaces also when considering $(n-2)$-clubs.
In Section \ref{sec:polsoficlub} we describe linearized polynomials that define the known families of clubs and in particular our description is complete when considering $(n-2)$-clubs, answering to an open problem posed in \cite{CsMPq5}.
Section \ref{sec:blocksetsKMarcs} is devoted to the study of blocking sets of R\'edei type naturally associated with $i$-clubs and to translation KM-arcs, for which we are able to provide classification results and non-equivalent constructions, as a consequence of the results described for $i$-clubs.
In Section \ref{sec:rankmetric} we deal with rank metric codes. Using the connection between $q$-systems and linear rank metric codes proved by Randrianarisoa in \cite{Randrianarisoa2020ageometric}, we are able to classify and construct linear rank metric codes with certain parameters. The interest in these codes is again due to the fact that the number of the non zero weights is only three, and for these kind of codes it seems very difficult to find examples.

\section{Preliminaries}\label{sec:prel}

\subsection{Linearized polynomials}

Let $q=p^h$, where $h$ is a positive integer and $p$ is a prime.
A \emph{linearized polynomial} (or a $q$-\emph{polynomial}) over $\F_{q^n}$ is a polynomial of the form
\[f(x)=\sum_{i=0}^t a_i x^{q^i},\]
where $a_i\in \F_{q^n}$ and $t$ is a non-negative integer.
Furthermore, if $a_t \neq 0$ we say that $t$ is the $q$-\emph{degree} of $f(x)$.
The set of all $q$-polynomials over $\F_{q^n}$ will be denoted by $\mathcal{L}_{n,q}$.
A well-known example of linearized polynomial is given by the trace, that is $\mathrm{Tr}_{q^n/q}(x)=x+x^q+\ldots+x^{q^{n-1}}$.

Sheekey in \cite{John}, in connection with linear sets and rank metric codes, gave the following definition.  Let $f(x) \in \mathcal{L}_{n,q}$. Then $f$ is said to be \emph{scattered} if for every $x,y \in \fqn^*$
\[ \frac{f(x)}{x}=\frac{f(y)}{y}\,\, \Leftrightarrow\,\, \frac{y}x\in \fq, \]
or equivalently
\[ \dim_{\fq}(\ker(f(x)-mx))\leq 1, \,\,\,\text{for every}\,\,\, m \in \fqn. \]
The term \emph{scattered} arises from the theory of linear sets; see \cite{BL2000}. Scattered polynomials have attracted a lot of attention because of their connections with several objects, such as maximum distance rank metric codes as we will see later.

For more details on linearized polynomials see \cite{wl}.

\subsection{Linear sets}\label{sec:linset}

In this paper, we deal with linear sets on the projective line and in some cases also on the projective plane.
More precisely, let $V$ be a $r$-dimensional $\fqn$-vector space and let $\Lambda=\PG(V,\fqn)=\PG(r-1,q^n)$. 
Let $U$ be an $\fq$-subspace of $V$, then the set
\[ L_U=\{\langle {u} \rangle_{\fqn} \colon {u}\in U\setminus\{{0}\}\} \]
is said to be an $\fq$-\emph{linear set} of rank $\dim_{\fq}(U)$.
The rank of $L_U$ will also be denoted by $\mathrm{Rank}(L_U)$.
The \emph{weight} of a projective subspace $S=\PG(W,\fqn) \subseteq \Lambda$ in $L_U$ is defined naturally as $w_{L_U}(S)=\dim_{\fq}(U \cap W)$.
In some cases, linearized polynomials can be used to describe linear sets. Indeed, suppose that $L_U$ is a linear set in $\PG(1,q^n)$ and it has rank $n$, then there exists a $q$-polynomial $f(x) \in \mathcal{L}_{n,q}$ such that $L_U$ is mapped by a collineation of $\mathrm{PGL}(2,q^n)$ into $L_f=L_{U_f}$ where
\[ U_f=\{ (x,f(x)) \colon x \in \fqn\}. \]
Nevertheless, it could be hard to find polynomials describing a fixed linear set.

Let now recall some basics relations between the size of a linear set, the number of point of a certain weight and its rank.
If $N_i$ denotes the number of points of $\Lambda$ having weight $i\in \{0,\ldots,k\}$  in $L_U$, the following relations hold:
\begin{equation}\label{eq:card}
    |L_U| \leq \frac{q^k-1}{q-1},
\end{equation}
\begin{equation}\label{eq:pesicard}
    |L_U| =N_1+\ldots+N_k,
\end{equation}
\begin{equation}\label{eq:pesivett}
    N_1+N_2(q+1)+\ldots+N_k(q^{k-1}+\ldots+q+1)=q^{k-1}+\ldots+q+1.
\end{equation}
Moreover, the following holds
\begin{equation}\label{eq:wpointsrank}
    w_{L_U}(P)+w_{L_U}(Q)\leq \mathrm{Rank}(L_U),
\end{equation}
for any $P,Q \in \mathrm{PG}(r-1,q^n)$ with $P\ne Q$.
An $\fq$-linear set $L_U$ is $\PG(1,q^n)$ for which all of its points have weight one is called a \emph{scattered} linear set. If all of its points have weight one except for one which has weight $i$, we call it $i$-\emph{club}.


By the above relations, it is easy to see that a scattered linear set of rank $k$ has $\frac{q^k-1}{q-1}$ points and an $i$-club of rank $k$ has size $q^{k-1}+\ldots+q^i+1$.

We refer to \cite{LavVdV} and \cite{Polverino} for comprehensive references on linear sets.

\subsection{Known examples of clubs}\label{sec:clubsexamp}

Up to now, very few examples of $i$-clubs are known, which have been found and summarized in \cite{DeBoeckVdV2016}, where the authors investigated such a family of linear sets in connection with KM-arcs; see Section \ref{sec:blocksetsKMarcs}.

The linear set $L_{\mathrm{Tr}_{q^n/q}(x)}$ is an example of $(n-1)$-club and in \cite[Theorem 3.7]{CsMP}, it has been proved that every $(n-1)$-club is $\mathrm{PGL}(2,q^n)$-equivalent to $L_{\mathrm{Tr}_{q^n/q}(x)}$.

A further important example is the following, which extends the above example. Let $n=\ell m$, $i=m(\ell-1)$, $\gcd(s,m)=1$ and $\sigma\colon x \in \fqn \mapsto x^{q^s} \in \fqn$. Then the linear set $L_T$ where 
\begin{equation}\label{eq:firstKMarc}
T(x)=\mathrm{Tr}_{q^{\ell m}/q^m}\circ \sigma(x) \in \mathcal{L}_{n,q}
\end{equation}
is an $i$-club in $\mathrm{PG}(1,q^{n})$ with $i=m(\ell-1)$.
Apart from these examples, polynomials defining $i$-clubs are not known except for \cite[Corollary 5.5]{BMZZ} for $n=4$ (see also \cite{CsZ2018}) and \cite[Corollary 6.3]{BMZZ}.
However, in \cite{DeBoeckVdV2016} two other examples of clubs were described.

\begin{construction}\label{constr:(n-2)clubpolynomial}
In \cite[Lemma 2.12]{DeBoeckVdV2016} it was proved that for any $\lambda \in \fqn^*$ such that $\{1,\lambda,\ldots,\lambda^{n-1}\}$ is an $\fq$-basis of $\fqn$, the $\fq$-linear set
\begin{equation}\label{eq:exKMnopol} L_{\lambda}=\{ \langle (t_1\lambda+\ldots+t_{n-1}\lambda^{n-1},t_{n-1}+t_n \lambda) \rangle_{\fqn} \colon (t_1,\ldots,t_n)\in \fq^n\setminus\{\mathbf{0}\} \},
\end{equation}
is an $(n-2)$-club of rank $n$ in $\PG(1,q^n)$.
\end{construction}

\begin{construction}\label{constr:clubviascattered}
In \cite[Lemma 3.6]{DeBoeckVdV2016} was detected the following $i$-clubs.
Let $n=rt$, $t,r>1$ and $f$ a scattered polynomial in $\mathcal{L}_{t,q}$. Let 
\begin{equation}\label{eq:Ua,b}U_{a,b}=\left\{ \left(f(x_0)-ax_0,bx_0+\sum_{i=1}^{r-1}x_i \omega^i\right) \colon x_i \in \F_{q^t} \right\},\end{equation}
for some fixed $a,b \in \F_{q^t}$ with $b\ne 0$ and $\{1,\omega,\ldots,\omega^{r-1}\}$ an $\F_{q^t}$-basis of $\F_{q^n}$. Then $L_{U_{a,b}}$ is an $i$-club, with
\[ i=\left\{ \begin{array}{ll} t(r-1), & \text{if}\,\, f(x)-ax\,\, \text{is invertible over}\,\,\F_{q^t},\\
t(r-1)+1, & \text{otherwise}.\end{array} \right. \]
\end{construction}

\subsection{Algebraic preliminaries}

In this section we will first recall the notion of dual basis and properties, then we will briefly describe the linear analogue of Cauchy-Davenport inequality and Vosper's theorem and some related results.

We start by recalling that the trace function from $\fqn$ over $\fq$ is a linear map that defines a nondegenerate symmetric bilinear form as follows:
\[
(a,b) \in \F_{q^n} \times \F_{q^n} \mapsto \mathrm{Tr}_{q^n/q}(ab) \in \F_q.
\]
This allows us to give the following definition.
Two ordered $\F_{q}$-bases $\mathcal{B}=(\xi_0,\ldots,\xi_{n-1})$ and $\mathcal{B}^*=(\xi_0^*,\ldots,\xi_{n-1}^*)$ of $\F_{q^n}$ are said to be \emph{dual bases} if 
\[
\mathrm{Tr}_{q^n/q}(\xi_i \xi_j^*)= 
\left\{\begin{array}{ll} 
1 & \text{if } i=j,\\
0 & \text{if } i\ne j.
\end{array}\right. 
\] 
For any $\F_q$-basis $\mathcal{B}=(\xi_0,\ldots,\xi_{n-1})$ there exists a unique dual basis $\mathcal{B}^*=(\xi_0^*,\ldots,\xi_{n-1}^*)$ of $\mathcal{B}$, see e.g.\ \cite[Definition 2.30]{lidl_finite_1997}. 

When $\mathcal{B}$ is a polynomial basis, then using the minimal polynomial of the element defining the polynomial basis over $\fq$, we can easily construct the relative dual basis.

\begin{corollary}\label{cor:dualbasis}\cite[Corollary 2.7]{NPSZ2022complwei}
Let $\lambda \in \fqn$ such that $\mathcal{B}=(1,\lambda,\ldots,\lambda^{n-1})$ is an ordered $\fq$-basis of $\fqn$.
Let $f(x)=a_0+a_1x+\ldots+a_{n-1}x^{n-1}+x^n$ be the minimal polynomial of $\lambda$ over $\fq$. 
Then the dual basis $\mathcal{B}^*$ of $\mathcal{B}$ is 
\[ \mathcal{B}^*=(\delta^{-1}\gamma_0,\ldots,\delta^{-1}\gamma_{n-1}), \]
where $\delta=f'(\lambda)$ and $\gamma_i=\sum_{j=1}^{n-i} \lambda^{j-1}a_{i+j}$, for every $i \in \{0,\ldots,n-1\}$. 
\end{corollary}

In the case in which the minimal polynomial of $\lambda$ has degree two or three, the dual basis has even a simpler description.

\begin{corollary}\label{cor:binetrin}\cite[Corollary 2.9]{NPSZ2022complwei}
Let $\lambda \in \fqn$ such that $\mathcal{B}=(1,\lambda,\ldots,\lambda^{n-1})$ is an ordered $\fq$-basis of $\fqn$.
\begin{itemize}
    \item If $f(x)=x^n-d$ is the minimal polynomial of $\lambda$ over $\fq$, then the dual basis of $\mathcal{B}$ is
    \[ \mathcal{B}^*=\left(\frac{\lambda^n}{nd},\frac{\lambda^{n-1}}{nd},\ldots,\frac{\lambda}{nd},\frac{1}{nd}\right). \]
    \item If $f(x)=x^n-cx^k-1$ is the minimal polynomial of $\lambda$ over $\fq$, then the dual basis of $\mathcal{B}$ is
    \[ \mathcal{B}^*=(\delta^{-1}\lambda^{k-1},\delta^{-1}\lambda^{k-2},\ldots,\delta^{-1},\delta^{-1}\lambda^{n-1},\ldots,\delta^{-1}\lambda^k), \]
    where $\delta=\frac{n\lambda^{n-1}-ck\lambda^{k-1}}{-c+\lambda^{n-k}}$.
\end{itemize}
\end{corollary}

We also need the following result, which states that every hyperplane in $\F_{q^n}$ admits a polynomial basis.

\begin{proposition} \label{lemm:formhyperplanes}
Let $S$ be an $(n-1)$-dimensional $\F_{q}$-subspace of $\F_{q^n}$. Let $\gamma \in \F_{q^n}^*$ such that $\F_{q}(\gamma)=\F_{q^n}$. Then there exists $c \in \F_{q^n}^*$ such that $S=c \langle 1,\gamma,\ldots,\gamma^{n-2} \rangle _{\F_q}$.
\end{proposition}
\begin{proof}
Since $\{1,\gamma,\ldots,\gamma^{n-1}\}$ is an $\fq$-basis of $\F_{q^n}$, then the $\fq$-subspace $T=\langle 1,\gamma,\ldots,\gamma^{n-2}\rangle_{\fq}$ is an hyperplane of $\fqn$. By \cite[Propositon 3.5.16]{combbook}, for any hyperlane $S$ of $\fqn$ there exists $c \in \F_{q^n}^*$ such that $S=cT$ and the assertion is proved.
\end{proof}

We now recall some algebraic results which will be used in the rest of the paper. The first two correspond to the linear analogue of Cauchy-Davenport inequality and to the linear analogue of the Vosper's Theorem.
Let denote by $S \cdot T$ the set of the products of any two elements in the subset $S$ and $T$ of $\fqn$.

\begin{theorem} \cite[Theorem 3]{BSZ2015} \label{teo:bachocserrazemor}
Let $S$ be an $\F_q$-subspace of $\F_{q^n}$. Then
\begin{itemize}
    \item[i)] either for every $\F_q$-subspace $T \subseteq \F_{q^n}$ we have
    \[
    \dim_{\F_q}(\langle S\cdot T \rangle_{\F_q})\geq \min\{ \dim_{\F_q}(S)+\dim_{\F_q}(T)-1,n\},
    \]
    \item[ii)] or there exists a positive integer $t>1$ that divides $n$, such that for every $\F_q$-subspace $T \subseteq \F_{q^n}$ satisfying
    \[
    \dim_{\F_q}(\langle S\cdot T \rangle_{\F_q})< \dim_{\F_q}(S)+\dim_{\F_q}(T)-1,
    \]
    we have that $\langle S \cdot T \rangle_{\F_q}$ is also an $\F_{q^t}$-subspace.
\end{itemize}
\end{theorem}

The pairs $(S,T)$ satisfying the equality in i) of Theorem \ref{teo:bachocserrazemor} with 
\[
    \dim_{\F_q}(\langle S\cdot T \rangle_{\F_q})= \dim_{\F_q}(S)+\dim_{\F_q}(T)-1,
    \]
are called \emph{critical pairs}. The classification of critical pairs is a hard problem in general and few results are known. The next result classify them when $n$ is a prime number.

\begin{theorem} \cite[Theorem 3]{BSZ2017} \label{teo:bachocserrazemor2}
Suppose $n$ is a prime number. Let $S,T$ be $\F_q$-subspaces of $\F_{q^n}$ such that $2 \leq \dim_{\F_q}(S),\dim_{\F_q}(T)$ and $\dim_{\F_q}(\langle S \cdot T \rangle_{\fq}) \leq n-2$. If
\[
\dim_{\F_q}(\langle S \cdot T \rangle_{\fq}) =\dim_{\F_q}(S)+\dim_{\F_q}(T)-1,
\]
then $S=g \langle 1,a,\ldots,a^{\dim_{\F_q}(S)-1}\rangle_{\F_q}$ and $T=g' \langle 1,a,\ldots,a^{\dim_{\F_q}(T)-1}\rangle_{\F_q}$, for some $g,g',a \in \F_{q^n}$.
\end{theorem}

Here we recall a result from \cite{NPSZ2022minsize}, whose consequence is the classification of critical pairs in the case in which one of the subspaces has dimension $2$; see \cite[Proposition 6.3]{NPSZ2022minsize}.

\begin{lemma}\label{lemma:power}\cite[Lemma 3.1]{NPSZ2022minsize}
Let $S$ be an $\fq$-subspace of $\fqn$ of dimension $k\geq2$ and let $\mu \in \fqn\setminus\fq$. Let $t=\dim_{\fq}(\fq(\mu))$.
\begin{itemize}
    \item [(a)] If $\dim_{\fq}(S\cap \mu S)=k$, then $S$ is an $\fq(\mu)$-subspace.
    \item [(b)] Suppose that $\dim_{\fq}(S\cap \mu S)=k-1$.
    \begin{itemize}
    \item [(b1)] If $t\geq k$ then $S=b \langle 1,\mu,\ldots,\mu^{k-1}\rangle_{\fq}$, for some $b \in \fqn^*$ and $t \neq k$.
    \item [(b2)] If $t\leq k-1$, write $k=t\ell+m$ with $m<t$, then $m>0$ and $S=\overline{S}\oplus b\langle 1,\mu,\ldots,\mu^{m-1}\rangle_{\fq}$, where $\overline{S}$ is an $\F_{q^t}$-subspace of dimension $\ell$, $b \in \fqn^*$ and $b \F_{q^t} \cap \overline{S}=\{0\}$.
    \end{itemize}
    In particular, $t$ divides $n$.
\end{itemize}
\end{lemma}

\section{Classification of $(n-2)$-clubs}\label{sec:classn-2}

In this section we are going to classify $(n-2)$--clubs in $\PG(1,q^n)$. We will first detect a larger family of linear sets which contains all the $(n-2)$--clubs (up to equivalence), then we will describe how to find the number of points of a certain weight in relation to the intersection of some subspaces. Using the above mentioned results together with the results on critical pairs over finite fields, we are able to completely classify $(n-2)$-clubs.

First observe that for $n=3$ all the $2$-clubs are $\mathrm{P\Gamma L}(2,q^3)$-equivalent to $L_{\mathrm{Tr}_{q^3/q}}$, for $n=4$ we only have two not $\mathrm{P\Gamma L}(2,q^4)$-equivalent $2$-clubs (that are $L_{x^q-x^{q^3}}$ and the one from Construction \ref{constr:(n-2)clubpolynomial}) and the only $3$-club is $L_{\mathrm{Tr}_{q^4/q}}$. So, in this section we will assume $n\geq 5$.

Let us start by studying the weight distribution of a class of linear sets which contains the clubs we are interested.

\begin{proposition} \label{prop:formclub}
Let $S$ be an $h$-dimensional $\F_q$-subspace of $\F_{q^n}$ such that $1 \in S$ and $h\leq n-2$. Let $(a,b) \in \F_{q^n}^2$, with $a \notin S$ and $b \notin \F_q$. Let 
\[
U=(S \times \{0\}) +\langle (1,1) \rangle_{\F_{q}} + \langle (a,b) \rangle_{\F_{q}}  \subseteq \F_{q^n}^2.
\]
Then 
\begin{equation} \label{eq:formclub}
    L_U=\{\langle (s+\alpha+\beta a,\alpha +\beta b) \rangle \colon s \in S, \alpha,\beta \in \F_q, (s,\alpha,\beta) \neq (0,0,0)\}
\end{equation}
is an $\fq$-linear set of rank $h+2$ in $\PG(1,q^n)$, with $w_{L_U}(\langle (1,0) \rangle_{\F_{q^n}})=h$ and $w_{L_U}(\langle(0,1) \rangle_{\F_{q^n}})=1$.
\end{proposition}
\begin{proof} Since $b \notin \F_q$, we have that $(a,b) \notin S\times \{0\} + \langle (1,1) \rangle_{\F_{q}}$ and so $U$ is a direct sum and $\dim_{\F_q}(U)=h+2$, that is $Rank (L_U)= h+2$. From $b \notin \F_q$, it also follows that $\alpha +\beta b=0$, if and only if $\alpha=\beta=0$. Thus, $w_{L_U}(\langle (1,0) \rangle_{\F_{q^n}})=h$.  Finally, $s+\alpha+\beta a=0$ implies $\beta=0$. Indeed, if $\beta \neq 0$, then $a= - \beta^{-1}(s+\alpha) \in S$, a contradiction.  So, $w_{L_U}(\langle(0,1) \rangle_{\F_{q^n}})=1$.
\end{proof}

Note that, since the point $\langle (1,0) \rangle_{\F_{q^n}}$ has weight $h$ in $L_U$ and the rank of $L_U$ is $h+2$, by \eqref{eq:wpointsrank} all the other points can have weight at most two in $L_U$. In the next result we show how to find the points of weight two.

\begin{theorem} \label{teo:existenceclub}
Let $S$ be an $h$-dimensional $\F_q$-subspace of $\F_{q^n}$ such that $1 \in S$ and $3 \leq h \leq n-2$. Consider 
\[
U=(S \times \{0\}) \oplus\langle (1,1) \rangle_{\F_{q}} \oplus \langle (a,b) \rangle_{\F_{q}}  \subseteq \F_{q^n}^2,
\]
with $a \notin S$, and $b \notin \F_q$. 
Then the set of points of weight $2$ in $L_U$ different from $\langle (1,0) \rangle_{\F_{q^n}}$ is
\[ \{ P_s=\langle (-s+a,b)\rangle_{\F_{q^n}} \colon s \in S \cap (a+bS) \} \]
and the size of such a set is $\lvert S \cap (a+bS) \rvert$. In particular, $L_U$ is a $h$-club of rank $h+2$ if and only if $a \notin S+bS$.
\end{theorem}

\begin{proof}
From Proposition \ref{prop:formclub}, $L_U$ is a linear set of rank $h+2$, with $w_{L_U}(\langle (1,0) \rangle_{\F_{q^n}})=h$ and $w_{L_U}(\langle(0,1) \rangle_{\F_{q^n}})=1$. First note that, by \eqref{eq:wpointsrank}, $w_{L_U}(P) \leq 2$ for every point $P$ different from $\langle (1,0) \rangle_{\F_{q^n}}$.
Let consider the map
\[
\begin{tabular}{l c c c }
$\Phi:$ & $S \cap (a+bS)$ & $\longrightarrow$ & $\{P \in L_U \colon w_{L_U}(P)=2, P \neq \langle (1,0) \rangle_{\F_{q^n}}\}$ \\
& $s$ & $\longmapsto$ & $P_s=\langle (-s+a,b)\rangle_{\F_{q^n}}$.
\end{tabular}
\]
First, we prove that $\Phi$ is well-defined. If $S \cap (a+bS)=\emptyset$, then $\Phi$ is not defined. Therefore, suppose that $S \cap (a+bS)\ne\emptyset$ and let $\overline{s} \in S \cap (a+bS)$, so $\overline{s}=a+b\overline{s}'$, for some $\overline{s}'\in S$. Then $(-\overline{s}+a,b)\in U$ and $(-\overline{s}+a,b)=b(-\overline{s}',1)=b(-1-\overline{s}'+1,1) \in U$. Since $b \notin \F_q$, the point $P_{\overline{s}}$ has weight 2 in $L_U$. \\
Now, we show that $\Phi$ is a one-to-one correspondence. 
Indeed, we only need to prove that $\Phi$ is surjective. If $L_U$ has no point of weight 2, it is clear. Otherwise, let $P$ be a point with weight $2$ in $L_U \setminus \{\langle (1,0) \rangle_{\F_{q^n}}\}$. So \[
P=\langle (s_1+\alpha_1+\beta_1a,\alpha_1+\beta_1 b)\rangle_{\F_{q^n}}=\langle (s_2+\alpha_2+\beta_2a,\alpha_2+\beta_2 b)\rangle_{\F_{q^n}},\] 
for some $s_1,s_2 \in S$ and $\alpha_1,\alpha_2,\beta_1,\beta_2 \in \F_q$,
with $(s_1+\alpha_1+\beta_1a,\alpha_1+\beta_1 b)$ and $(s_2+\alpha_2+\beta_2a,\alpha_2+\beta_2 b)$ $\F_{q^n}$-linearly dependent by a scalar $t\in \F_{q^n} \setminus \F_q$, that is $(s_1+\alpha_1+\beta_1a,\alpha_1+\beta_1 b)=t(s_2+\alpha_2+\beta_2a,\alpha_2+\beta_2 b)$. This implies 
\[
(s_1+\alpha_1+\beta_1 a)(\alpha_2+\beta_2 b)=(s_2+\alpha_2+\beta_2 a)(\alpha_1+\beta_1 b),
\]
and so 
\begin{equation} \label{eq:converslemmapoint2club1}
    s_1 \alpha_2-s_2 \alpha_1=a(\alpha_1 \beta_2-\alpha_2 \beta_1)-b(\alpha_1 \beta_2-\alpha_2 \beta_1)+b(s_2\beta_1-s_1\beta_2)
\end{equation}
Now, we observe that $\gamma=\alpha_1\beta_2-\alpha_2\beta_1 \neq 0$, that is $(\alpha_1,\beta_1)$ and $(\alpha_2,\beta_2)$ are $\fq$-linearly independent. Indeed, suppose that $(\alpha_1,\beta_1)$ and $(\alpha_2,\beta_2)$ are $\F_q$-proportional then there exists $\lambda \in \F_q$ such that $\alpha_2+\beta_2 b=\lambda (\alpha_1+\beta_1 b)=\lambda t (\alpha_2+\beta_2 b)$. This implies that $\lambda t=1$ and so $t \in \F_q$, a contradiction. \\
Therefore \eqref{eq:converslemmapoint2club1} implies 
\begin{equation} \label{eq:converslemmapoint2club2}
s_1 \frac{\alpha_2}{\gamma}-s_2 \frac{\alpha_1}{\gamma}=a+b\left(s_2 \frac{\beta_1}{\gamma}-s_1 \frac{\beta_2}{\gamma}-1 \right).
\end{equation}
 Let $\overline{s}=s_1 \frac{\alpha_2}{\gamma}-s_2 \frac{\alpha_1}{\gamma}$, because of $s_2 \frac{\beta_1}{\gamma}-s_1 \frac{\beta_2}{\gamma}-1 \in S$. By  \eqref{eq:converslemmapoint2club2}, $\overline{s} \in S \cap (a+b S)$. Since $(s_1+\alpha_1+\beta_1a,\alpha_1+\beta_1 b),(s_2+\alpha_2+\beta_2a,\alpha_2+\beta_2 b) \in U$ then 
 \[
 \frac{\alpha_2}{\gamma} (s_1+\alpha_1+\beta_1a,\alpha_1+\beta_1 b) -\frac{\alpha_1}{\gamma}(s_2+\alpha_2+\beta_2a,\alpha_2+\beta_2 b)=(\overline{s}-a,-b)
 \]
belongs to $U$ and defines again the point $P$ and hence $\Phi$ is surjective. So, $P= \langle(\overline{s}-a,-b) \rangle_{\F_{q^n}}=\Phi({-\overline{s}})$. In particular, if there exists a point of weight 2 in $L_U$ then $S \cap (a+bS) \neq \emptyset$. \\
Therefore, $L_U$ has only one point of weight greater than $1$ if and only if $S \cap (a+bS)= \emptyset$, which happens if and only if $a \notin S+bS$.
This completes the proof.
\end{proof}

Also in the case in which the linear sets as in Theorem \ref{teo:existenceclub} are not $(h-2)$-club, they have an interesting weight distribution if the points and size.

\begin{corollary} \label{cor:cardcasenotclub}
Let $S$ be an $h$-dimensional $\F_q$-subspace of $\F_{q^n}$ such that $1 \in S$ and $3 \leq h \leq n-2$. Consider 
\[
U=(S \times \{0\}) \oplus\langle (1,1) \rangle_{\F_{q}} \oplus \langle (a,b) \rangle_{\F_{q}}  \subseteq \F_{q^n}^2,
\]
with $a \notin S$, and $b \notin \F_q$. If $a \in S+bS$, then $L_U$ has $q^j$ points of weight $2$, with \[j=\dim_{\F_q}(S \cap bS)=2h-\dim_{\F_q}(S + bS)\]
and 
\[
\lvert L_U \rvert =q^{h+1}+q^{h}-q^{j+1}+1.
\]
\end{corollary}
\begin{proof}
Since $a \in S +bS$, we have that $S \cap (a+bS) \neq \emptyset$. So, by Theorem \ref{teo:existenceclub}, we know that that number of points of weight 2 is $\lvert S \cap (a+bS) \rvert$. Finally by \eqref{eq:pesicard} and \eqref{eq:pesivett}, we get the assertion.
\end{proof}

Next result shows that $h$-clubs of rank $h+2$  (up to the action of $\PGL(2, q^n)$ on $\PG(1, q^n)$) are as described in \eqref{eq:formclub}.

\begin{lemma}\label{lem:canonicalform}
If $L_W$ is an $h$-club of rank $h+2 \leq n$ in $\PG(1,q^n)$. Then, $L_W$ is $\PGL(2,q^n)$-equivalent to a linear set $L_U$, where $U=(S \times \{0\}) \oplus\langle (1,1) \rangle_{\F_{q}} \oplus \langle (a,b) \rangle_{\F_{q}}$, for some $h$-dimensional $\F_q$-subspace $S$ containing $1$, with $a \notin S$, $b \notin \F_q$ and $a \notin S+bS$.
\end{lemma}
\begin{proof}
Let $P=\langle \mathbf{v}_1 \rangle_{\F_{q^n}} \in L_W$ be the point of weight $h$, with $\mathbf{v}_1 \in W$. Let $Q= \langle \mathbf{v}_2 \rangle_{\F_{q^n}} \in L_W \setminus \{P\}$, with $\mathbf{v}_2 \in W$ and let $\psi \in \PGL(2,q^n)$ be the projectivity of $\PG(1, q^n)$ induced by the unique $\F_{q^n}$-isomorpishm $f$ of $\F_{q^n}^2$ that maps $\mathbf{v}_1$ in $(1,0)$ and $\mathbf{v}_2$ in $(1,1)$. Then $\psi(L_W)=L_{U}$ where $U=f(W)$ and $(1,0),(1,1) \in U$. Clearly, $\psi(P)=\langle(1, 0) \rangle_{\F_{q^n}}$, $\psi(Q)=\langle (1, 1) \rangle_{\F_{q^n}}$ and  $w_{L_U}(\langle (1,0) \rangle_{\F_{q^n}})=h$.
Let $S'=U \cap \langle (1, 0) \rangle_{\F_{q^n}}$. We can write $S'=S \times \{0\}$, where $S$ is an $\fq$-subspace of $\fqn$, $\dim_{\F_q}(S')=\dim_{\F_q}(S)=h$ and $1 \in S$. Let $(a,b) \in U$ such that $U= (S \times \{0\}) \oplus\langle (1,1) \rangle_{\F_{q}} \oplus \langle (a,b) \rangle_{\F_{q}}$, so 
\begin{equation}\label{eq:subspaceU}
U=\{ (s+\alpha+\beta a,\alpha +\beta b)  \colon s \in S, \alpha,\beta \in \F_q\}.
\end{equation} 
Since $(a,b) \notin (S \times \{0\}) \oplus\langle (1,1) \rangle_{\F_{q}}$, then either $a \notin S$ or $b \notin \F_q$. Suppose that $b \in \F_q$. Then $(-b+a,0) \in U$ (choosing $s=0$, $\alpha=-b$ and $\beta=1$ in \eqref{eq:subspaceU}). So $-b+a \in S$ and then $a \in S$ (since $b\in \fq$), a contradiction. Therefore, we have that $b \notin \F_q$. Suppose now that $a \in S$. Then, (choosing $s=-a$, $\alpha=0$, $\beta=1$ in \eqref{eq:subspaceU}) the point  $\langle (0,b) \rangle_{\F_{q^n}}=\langle (0,1) \rangle_{\F_{q^n}}$ is a point different from $\langle (1,0) \rangle_{\F_{q^n}}$ that has weight $2$ in $L_U$, this is a contradiction since $L_U$ is an $h$-club. Finally, by Theorem \ref{teo:existenceclub}, we get $a \notin S+bS$.
\end{proof}

\begin{remark}
Note that if $\dim_{\F_q}(S)=h < n/2$, then for every $b \in \F_{q^n} \setminus \F_q$, $\dim_{\F_q}(S+bS)<n$ and hence there exists $a \in \F_{q^n} \setminus (S + bS)$. Therefore, the $\fq$-subspace $U=(S \times \{0\}) \oplus\langle (1,1) \rangle_{\F_{q}} \oplus \langle (a,b) \rangle_{\F_{q}}$ defines an $h$-club of rank $h+2$.
\end{remark}

Our aim is to give a complete classification of $(n-2)$-club of rank $n$. By Lemma \ref{lem:canonicalform}, we may assume that $(n-2)$-clubs are as described in \eqref{eq:formclub} and we start to prove under which conditions on $S$, $a$ and $b$ linear sets as in Lemma \ref{lem:canonicalform} are clubs.

\begin{theorem}\label{th:classclass}
Let $S$ be an $(n-2)$-dimensional $\F_q$-subspace of $\F_{q^n}$, such that $1 \in S$. Consider 
\[
U=(S \times \{0\}) \oplus\langle (1,1) \rangle_{\F_{q}} \oplus \langle (a,b) \rangle_{\F_{q}}  \subseteq \F_{q^n}^2,
\]
with $a \notin S$ and $b\notin {\F_{q}}$. Let $T=\langle 1, b \rangle_{\F_q}$ and $\F_{q^t}=\F_q(b)$.
Then only one of the following three cases occurs:
\begin{enumerate}
    \item $\dim_{\F_q}(\langle S \cdot T \rangle_{\fq})=n$ and $L_U$ is not an $(n-2)$-club. In this case $L_U$ contains $q^{n-4}$ points of weight 2 and $\lvert L_U \rvert=q^{n-1}+q^{n-2}-q^{n-3}+1$.
    \item $\dim_{\F_q}(\langle S \cdot T \rangle_{\fq})=n-1$, and in this case $L_U$ is an $(n-2)$-club of rank $n$ if and only if $a \notin \langle S \cdot T \rangle_{\fq}$. So, if $a \in \langle S \cdot T \rangle_{\fq}$, $L_U$ is an $\F_q$-linear set having $q^{n-3}$ points of weight 2 and $\lvert L_U \rvert=q^{n-1}+1$. While if $a \notin \langle S \cdot T \rangle_{\fq}$:
    \begin{enumerate}[(2.1)]
        \item if $t\geq n-2$ there exists $c \in \F_{q^n}^*$, such that $S=c \langle 1,b,\ldots,b^{n-3}\rangle_{\F_{q}}$,
    \[
    L_U=\left\{\left\langle \left(\sum_{i=0}^{n-3}cx_ib^i+\alpha+\beta a,\alpha +\beta b\ \right) \right\rangle_{\F_{q^n}} \colon x_i,\alpha,\beta \in \F_{q} \mbox{ not all zero}\right\},
    \]
    and $t=n$.
        \item if $t \leq n-3$, there exist $\ell \in \mathbb{N}, c \in \F_{q^n}^*$ such that $n=t(\ell+1)$, with $t \geq 3$ $S=\overline{S}\oplus c \langle 1, b,\ldots,b^{t-3} \rangle_{\F_{q}}$, with $\overline{S}$ an $\F_{q^t}$-subspace of dimension $\ell$ such that $c\F_{q^t} \cap \overline{S}=\{0\}$, and 
    \[
    L_U=\left\{\left\langle \left(\overline{s}+\sum_{i=0}^{t-3}cx_ib^i+\alpha+\beta a,\alpha +\beta b \right) \right\rangle_{\F_{q^n}} \colon x_i,\alpha,\beta \in \F_{q},\overline{s} \in \overline{S} \mbox{ not all zero}\right\}.
    \]
    \end{enumerate}
    \item $\dim_{\F_q}(\langle S \cdot T \rangle_{\fq})=n-2$, then $n$ is even, $\F_q(b)=\F_{q^2}$, $S$ is an $\F_{q^2}$-subspace. Moreover, there exists $c \in \F_{q^n}^*$ such that $S=c\langle 1, \gamma,\ldots,\gamma^{n/2-2}\rangle_{\F_{q^2}}$ and
    \[
    L_U=\left\{\left\langle \left(\sum_{i=0}^{n/2-2}cx_i\gamma^i+\alpha+\beta a,\alpha +\beta b\ \right) \right\rangle_{\F_{q^n}} \colon x_i \in \F_{q^2},\alpha,\beta \in \F_{q} \mbox{not all zero}\right\},
    \] for some $\gamma \in \F_{q^n}^*$ such that $\F_{q^2}(\gamma)=\F_{q^n}$. Moreover, $L_U$ is an $(n-2)$-club of rank $n$ if and only if $a \notin c\langle 1,\gamma,\ldots,\gamma^{n/2-2} \rangle_{\F_{q^2}}$. If $a \in c\langle 1,\gamma,\ldots,\gamma^{n/2-2} \rangle_{\F_{q^2}}$, then all the points of $L_U$ different from $\langle (1,0) \rangle_{\F_{q^n}}$ have weight 2 and $\lvert L_U \rvert =q^{n-2}+1$.
\end{enumerate}
In particular, if $n$ is prime Cases (2.2) and (3) cannot occur.
\end{theorem}
\begin{proof}
\textbf{Case (1)}: This case is a consequence of Theorem \ref{teo:existenceclub}. Indeed, note that $\langle S\cdot T\rangle_{\fq}=S+bS=\F_{q^n}$ and hence $a \in S+bS$, so that Theorem \ref{teo:existenceclub} implies that $L_U$ is not an $(n-2)$-club. Also, by Corollary \ref{cor:cardcasenotclub} the weight distribution and the size of $L_U$ is determined.\\
\textbf{Case (2)}: In the case that $\dim_{\F_q}(\langle S \cdot T \rangle)=n-1$, we have $\dim_{\F_q}(S\cap bS)=n-3$, and so by Lemma \ref{lemma:power}, we need to analyze the following two cases:
\begin{itemize}
    \item if $t \geq n-2$, there exists $c \in \F_{q^n}$ such that $S=c \langle 1,b,\ldots,b^{n-3} \rangle_{\F_q}$. Moreover, since $n>4$, we have $t>n-2$ and so $t=n$.
    \item If $t < n-2$, we write $n-2=t \ell+m$, with $0<m<t$. In this case, since $t \mid n$ then $t \mid m+2$ and so there exists $s \in \mathbb{N}$ such that $m+2=st$. Therefore, since $m<t$ then $2>t(s-1)$. Because of $t\geq 2$, $s=1$, hence $n=t(\ell +1 )$. Also since $t \mid m+2$ and $0<m<t$, we have $t \geq 3$.
\end{itemize}
\textbf{Case (3)}: Note that $\dim_{\F_q}(\langle S \cdot T \rangle_{\fq})=n-2$ implies $S \cap bS=S$. From (a) of Lemma \ref{lemma:power}, we get that $S$ is an $\fq(b)$-subspace and hence $t \mid n-2$. Since $t \mid n$ then $t\mid 2$ and $t=2$. Hence $\dim_{\F_{q^2}}(S)=n/2-1$ and fixed $\gamma \in \F_{q^n}^*$ such that $\F_{q^n}=\F_{q^2}(\gamma)$, we can now apply Proposition \ref{lemm:formhyperplanes} to obtain the desired form. In the case $a \in c\langle 1,\gamma,\ldots,\gamma^{n/2-2} \rangle_{\F_{q^2}}$ we have that $a \in S+bS=S$ and so by Corollary \ref{cor:cardcasenotclub} the weight distribution and the size of $L_U$ are determined.

Finally, suppose that $n$ is a prime number. Then only the first case of Theorem \ref{teo:bachocserrazemor} can occur, which means that $\dim_{\F_q}(\langle S \cdot T \rangle_{\fq}) \in \{n-1,n\}$ and correspond to Cases (1) and (2.1).
\end{proof}

As a consequence of Lemma \ref{lem:canonicalform} and Theorem \ref{th:classclass}, we get the following theorem of classification on $(n-2)$--clubs.

\begin{corollary} \label{cor:classn-2club}
Let $L$ be a linear set of rank $n$ of $\PG(1,q^n)$. Then $L$ is an $(n-2)$--club if and only if it is $\PGL(2,q^n)$--equivalent to a linear set $L_U$ such that $U$ has the following form:
$n=t(\ell+1)$, with $\ell \in \mathbb{N}$, and
\[U=\left\{ \left(\overline{s}+\sum_{i=0}^{t-3}cx_ib^i+\alpha+\beta a,\alpha +\beta b\ \right)  \colon x_i,\alpha,\beta \in \F_{q},\overline{s} \in \overline{S} \right\},\] 
with $a,b,c \in \F_{q^n}^*$, $\overline{S} \subseteq \F_{q^n}$ such that
\begin{itemize}
       \item $\F_q(b)=\F_{q^t}$,
        \item $\overline{S}$ is an
       $\F_{q^t}$-subspace of dimension $\ell$ such that $c\F_{q^t} \cap \overline{S}=\{0\}$,
       \item $1 \in \overline{S} \oplus c\langle 1,b,\ldots,b^{t-2} \rangle_{\F_q}$,
        \item $a \notin \overline{S}\oplus c\langle 1,b,\ldots,b^{t-2} \rangle_{\F_q} $.
\end{itemize} 
\end{corollary}
\begin{proof}
Theorem \ref{th:classclass} gives a complete characterization of subspaces defining $(n-2)$-clubs. So, to prove the assertion, it is enough to observe that the subspace $U$ in Case (2.1) of Theorem \ref{th:classclass} is obtained choosing $\ell=0$, Case (2.2) choosing $3 \leq t \leq n-3$ and finally Case (3) choosing $t=2$.
\end{proof}

The known examples of $(n-2)$-clubs are described in the following remarks.

\begin{remark}\label{rk:constr1}
Let $U$ be defined as follows
\[U=\left\{ \left(\sum_{i=0}^{n-3}cx_ib^i+\alpha+\beta a,\alpha +\beta b\right) \colon x_i,\alpha,\beta \in \F_{q} \right\},\] 
with $a,b,c \in \F_{q^n}^*$ such that
    \begin{itemize}
        \item $\F_q(b)=\F_{q^n}$,
        \item $a \notin c\langle 1, b,\ldots,b^{n-3},b^{n-2}\rangle_{\F_{q}}$.
        \end{itemize}
Choosing $c=1$, $a=b-b^{n-1}$ and by applying the projectivity of $\PG(1,q^n)$ induced by the matrix 
$
\left(
\begin{matrix} 
b & -a \\
0 & 1
\end{matrix}
\right)
$, we obtain the family of linear sets described in Construction \ref{constr:(n-2)clubpolynomial}.
\end{remark}

\begin{remark} \label{vandervoorderm:n-2club}
The families of $i$-clubs described in Construction \ref{constr:clubviascattered} give $(n-2)$-clubs in the following cases
\begin{itemize}
    \item $n=2r$, $r>1$, $f(x)=a_0x+a_1x^q \in \F_{q^2}[x]$ with $a_1\ne 0$ and
   \[U_{a,b}=\left\{ \left(a_0x_0+a_1x_0^q-ax_0,bx_0+\sum_{i=1}^{t-1}x_i \omega^i\right) \colon x_i \in \F_{q^2} \right\},\]
    for some fixed $a,b \in \F_{q^2}$ with $b\ne 0$ and $\{1,\omega,\ldots,\omega^{t-1}\}$ an $\F_{q^2}$-basis of $\F_{q^n}$. Then $L_{U_{a,b}}$ is an $(n-2)$-club if $a_0x+a_1x^q-ax$ is invertible over $\F_{q^2}$ 
   \item $n=3r$, $r>1$, $f(x)=a_0x+a_1x^q+a_2x^{q^2} \in \F_{q^3}[x]$ such that either $a_1= 0$ and $a_2 \ne 0$ or $a_1 \ne 0$ and $\N_{q^3/q}(a_2/a_1)\ne 1$, and
   \[U_{a,b}=\left\{ \left(a_0x_0+a_1x_0^q+a_2x_0^{q^2}-ax_0,bx_0+\sum_{i=1}^{t-1}x_i \omega^i\right) \colon x_i \in \F_{q^3} \right\},\]
    for some fixed $a,b \in \F_{q^3}$ with $b\ne 0$ and $\{1,\omega,\ldots,\omega^{t-1}\}$ an $\F_{q^3}$-basis of $\F_{q^n}$. Then $L_{U_{a,b}}$ is an $(n-2)$-club if $a_0x+a_1x^q+a_2x^{q^2}-ax$ is not invertible over $\F_{q^3}$.
    \end{itemize}
\end{remark}

\begin{remark}
The condition on being scattered on the polynomial $f$ for the subspaces in \eqref{eq:Ua,b} has been written down explicitly.
Indeed, linearized polynomials which results scattered in $\F_{q^2}$ and in $\F_{q^3}$ are exactly those that define linear sets of pseudoregulus type.
In particular, for the case of $\F_{q^3}$, $f$ is scattered if and only if there not exists an element $\alpha \in \F_{q^3}$ such that
\[ \dim_{\fq}(\ker(f(x)-\alpha x))=2. \]
By \cite[Theorem 2.24]{lidl_finite_1997}, this is possible if and only if
\[ f(x)-\alpha x = \gamma \mathrm{Tr}_{q^3/q}(\beta x), \]
for some nonzero $\beta,\gamma \in \F_{q^3}$.
Clearly, this can happen if and only if $a_1\ne 0$ and $\N_{q^3/q}(a_2/a_1)=1$.
\end{remark}

\section{Equivalence of clubs}\label{sec:equivalence}

In this section we will deal with the $\mathrm{\Gamma L}$-equivalence issue for the known families of clubs.
We start by studying when two examples of Construction \ref{constr:clubviascattered} are $\GammaL (2,q^n)$-equivalent.

\begin{theorem}\label{th:equivscattclub}
Let $n=rt$, $t,r>1$, $f_1,f_2\in \mathcal{L}_{t,q}$ such that $f_1$ and $f_2$ are scattered $q$-polynomials. 
Let 
\[U_1=\left\{ \left(f_1(x_0)-a_1x_0,b_1x_0+\sum_{i=1}^{r-1}x_i \omega_1^i\right) \colon x_i \in \F_{q^t} \right\},\]
and 
\[U_2=\left\{ \left(f_2(x_0)-a_2x_0,b_2x_0+\sum_{i=1}^{r-1}x_i \omega_2^i\right) \colon x_i \in \F_{q^t} \right\},\]
for some fixed $a_1,a_2,b_1,b_2 \in \F_{q^t}$ with $b_1,b_2 \ne 0$ and $\{1,\omega_1,\ldots,\omega_1^{r-1}\}$ and $\{1,\omega_2,\ldots,\omega_2^{r-1}\}$ two  $\F_{q^t}$-basis of $\F_{q^n}$. Then $U_1$ and $U_2$ are $\GammaL(2,q^n)$-equivalent if and only if the spaces $\overline{U}_{f_1}=\{(f_1(x_0)-a_1x_0,b_1x_0) : x_0 \in \F_{q^t}\}$ and $\overline{U}_{f_2}=\{(f_2(x_0)-a_2x_0,b_2x_0) : x_0 \in \F_{q^t}\}$ are $\GammaL(2,q^t)$-equivalent via an element $\phi \in \GammaL(2,q^t)$ such that $\phi$ fixes the subspace $\langle (0,1) \rangle_{\F_{q^t}}$.
\end{theorem}
\begin{proof}
Denote $\overline{S}_1=\langle \omega_1,\ldots,\omega_1^{r-1} \rangle_{\F_{q^t}}$ and $\overline{S}_2=\langle \omega_2,\ldots,\omega_2^{r-1} \rangle_{\F_{q^t}}$. Suppose that $U_1$ and $U_2$ are $\GammaL(2,q^n)$-equivalent. Then there exists a matrix $\begin{pmatrix} A & B \\ C & D \end{pmatrix} \in \mathrm{GL}(2,q^n)$ and $\rho \in \Aut(\F_{q^n})$ such that
\begin{equation} \label{eq:equivalenceiclubdeboeck}
\begin{pmatrix} A & B \\ C & D \end{pmatrix} U_1^{\rho}=U_2.
\end{equation}
Since $\langle (0,1) \rangle_{\F_{q^n}}$ is the only point in $L_{U_1}$ and $L_{U_2}$ having weight greater than 1, then $B = 0, D \neq 0$ and 
\[
D \overline{S}_1^{\rho}=\overline{S}_2.
\]
Moreover, \eqref{eq:equivalenceiclubdeboeck} is satisfied if and only if for every $x_0 \in \F_{q^t}$ and $ s_1 \in \overline{S}_1 $ there exist $y_0 \in \F_{q^t}$ and $ s_2 \in \overline{S}_2$ such that 
\begin{equation} \label{eq:reductionequivalence1}
C(f_1^{\rho}(x^{\rho}_0)-a_1^{\rho}x^{\rho}_0) +D(b^{\rho}_1x^{\rho}_0+s^{\rho}_1)=b_2 y_0+s_2
\end{equation}
and 
\begin{equation} \label{eq:reductionequivalence2}
A(f_1^{\rho}(x^{\rho}_0)-a_1^{\rho}x^{\rho}_0)=f_2(y_0)-a_2 y_0.
\end{equation}
Note that $\{1,\omega_2,\ldots,\omega_2^{r-1}\}$ is an $\F_{q^t}$-basis of $\F_{q^n}$, this allows us to write  $C=\sum_{i=0}^{r-1}C_i \omega_2^i$ and 
$D=\sum_{i=0}^{r-1}D_i \omega_2^i$, with $C_i$'s and $D_i$'s in $\F_{q^t}$.  Now,  \eqref{eq:reductionequivalence1} reads as follows
\[ \sum_{i=0}^{r-1} C_i \omega_2^i(f_1^{\rho}(x^{\rho}_0)-a_1^{\rho}x^{\rho}_0)+\sum_{i=0}^{r-1} D_i \omega_2^i(b^{\rho}_1x^{\rho}_0+s^{\rho}_1)=b_2 y_0+s_2 \]
and since $Ds^{\rho}_1 \in S_2$, the above equality implies 
\begin{equation} \label{eq:reductionequivalence3} 
C_0(f_1^{\rho}(x^{\rho}_0)-a_1^{\rho}x^{\rho}_0) +D_0(b^{\rho}_1x^{\rho}_0)=b_2 y_0.
\end{equation}
Note also that $A \in \F_{q^t}$, because of Equation \eqref{eq:reductionequivalence2}. Moreover, note that $(C_0,D_0) \neq (0,0)$. Suppose that $D_0 = 0$, then \eqref{eq:reductionequivalence3} together with \eqref{eq:reductionequivalence2} implies that $f_2(y_0)=\frac{A}{C_0}b_2y_0+a_2y_0$, for every $y_0 \in \F_{q^t}$. So $f_2$ is not a scattered $q$-polynomial, a contradiction. So $D_0 \neq 0$.
This means that the spaces $\overline{U}_{f_1}$ and $\overline{U}_{f_2}$ are $\GammaL(2,q^t)$-equivalent via the map of  $\GammaL(2,q^t)$ defined by the matrix $\begin{pmatrix}
A_0 & 0 \\
C_0 & D_0 
\end{pmatrix}$ and the restriction of the automorphism $\rho$ to $\F_{q^t}$. Conversely, suppose that $\overline{U}_{f_1}$ and $\overline{U}_{f_2}$ are $\GammaL(2,q^t)$-equivalent via an element $\phi \in \GammaL(2,q^t)$ such that $\phi (\langle (0,1) \rangle_{\F_{q^t}})=\langle (0,1) \rangle_{\F_{q^t}}$. Then $\phi$ can be represented by a matrix $\begin{pmatrix} A_0 & 0 \\ C_0 & D_0 \end{pmatrix} \in \mathrm{GL}(2,q^t)$ and $\rho \in \Aut(\F_{q^t})$ such that
\begin{equation} \label{eq:equivalenceiclubdeboeck2}
\begin{pmatrix} A_0 & 0 \\ C_0 & D_0 \end{pmatrix} \overline{U}_{f_1}^{\rho}=\overline{U}_{f_2}.
\end{equation}
Moreover, $\overline{S}^{\rho}_1$ is an $(r-1)$-dimensional $\F_{q^t}$-subspace of $\F_{q^{n}}$. By Lemma \ref{lemm:formhyperplanes} there exists $D' \in \F_{q^n}^*$ such that $D' \overline{S}_1^{\rho}=\overline{S}_2$. Let write $D'$ as an $\F_{q^t}$-linear combination of $\{1,\omega_2,\ldots,\omega_2^{r-1}\}$, namely $D'=\sum_{i=0}^{r-1}D_i' \omega_2^i$. Then we also have $\frac{D_0}{D_0'}D'\overline{S}_1^{\rho}=\overline{S}_2$ holds. Denote by $D=\frac{D_0}{D_0'}D'$. Therefore, since $D=D_0+\sum_{i=1}^{r-1} \frac{D_i'}{D_0'}\omega_2^i$ and $b_1^\rho x_0^\rho \sum_{i=1}^{r-1} \frac{D_i'}{D_0'}\omega_2^i \in \overline{S}_2$, then the image of $U_1$ via the invertible map $\Phi$ represented by the matrix $\begin{pmatrix}
A_0 & 0 \\
C_0 & D
\end{pmatrix}$ and the automorphism $\rho$ (whose action is extended to $\fqn$) is contained in $U_2$. Since $U_1$ and $U_2$ have the same dimension we obtain that $\Phi(U_1)=U_2$ and hence $U_1$ and $U_2$ are $\GammaL(2,q^n)$-equivalent.
\end{proof}

Due to the above result, we know that there are several $\GammaL (2,q^n)$-inequivalent examples of $i$-club in $\PG(1,q^n)$, we refer to Table 1 provided in \cite{NPSZ2022complwei} for a list of all the known families of scattered polynomials (for the more recent examples see \cite{LZ2,LMTZ,NSZ}). 

\begin{remark}
The number of inequivalent subspaces of the recent family associated with the scattered polynomials presented in \cite{NSZ}, which extends those in \cite{LZ2,LMTZ}, is computed in \cite[Theorem 4.12]{NSZ} and a lower bound is given in \cite[Theorem 4.14]{NSZ}.
For the number of inequivalent subspaces associated with the more classical family of Lunardon-Polverino polynomials see \cite{AutLunPol}.
\end{remark}

Now we proceed by showing that the two different types (according to the behaviour of $\dim_{\fq}(\langle S\cdot T\rangle_{\fq})$, cfr. Theorem \ref{th:classclass}) of $\fq$-subspaces obtained in Theorem \ref{cor:classn-2club} are $\Gamma\mathrm{L}(2,q^n)$-inequivalent. To this aim the following lemma is needed.

\begin{lemma} \label{lem:sisjinequivalent}
Consider the following $\fq$-subspaces of $\F_{q^n}$ and let
\begin{enumerate}
    \item \[S_1=\left\{ \sum_{i=0}^{n-3}c_1x_ib_1^i\colon x_i \in \F_{q} \right\},\] with $b_1,c_1 \in \F_{q^n}^*$ such that $\F_q(b_1)=\F_{q^n}$,
  
    \item if $n=t(\ell+1)$, with $\ell \in \mathbb{N}$,
    \[S_2=\left\{ \overline{s}+\sum_{i=0}^{t-3}c_2x_ib_2^i \colon x_i\in \F_q,\overline{s} \in \overline{S} \right\},\] with $b_2,c_2 \in \F_{q^n}^*$ such that $\F_q(b_2)=\F_{q^t}$, with $3 \leq t \leq n-3$, and  $\overline{S} \subseteq \F_{q^n}$ is an $\F_{q^t}$-subspace of dimension $\ell$ and $c_2\F_{q^t} \cap \overline{S}=\{0\}$.
    \item if $n$ is even,
    \[
    S_3=\left\{ \sum_{i=0}^{n/2-2}c_3x_i b_3^i \colon x_i \in \F_{q^2}\right\},
    \]
    with $c_3,b_3 \in \F_{q^n}^*$, such that $\F_{q^2}(b_3)=\F_{q^n}$.
\end{enumerate}
Then there cannot exist $\lambda \in \F_{q^n}^*$ such that $S_i = \lambda S_j^{\rho}$, for some $\rho \in \Aut(\F_{q^n})$ and $i,j \in \{1,2,3\}$ such that $i\ne j$.
\end{lemma}
\begin{proof}
Let $n$ any positive integer for which we can consider $S_i$ and $S_j$ for $i,j \in \{1,2,3\}$ with $i\ne j$.
Suppose that there exists $\lambda \in \F_{q^n}^*$ such that $S_i = \lambda S_j^{\rho}$ with $\rho \in \Aut(\F_{q^n})$. Without loss of generality, we may assume that $i<j$ and $\rho$ is the identity (up to replace the $c_i$'s, $b_i$'s, $\gamma$ and $\overline{S}$) and hence $S_i=\lambda S_j$. Now, we give a case-by-case analysis:
\begin{itemize}
\item[(i=1,j=2)] Here, we have that $ \lambda \overline{S} \subseteq S_1 $. It follows that $S_1^\perp \subseteq (\lambda \overline{S})^\perp$ and $\dim_{\fq}((\lambda\overline{S})^\perp)=n-\ell t=t$.
So, $(\lambda \overline{S})^\perp$ is an $\F_{q^t}$-subspace of $\fqn$ of dimension one.
Consider the ordered basis $\mathcal{B}=(1,b_1,\ldots,b_1^{n-1})$ and its dual basis $\mathcal{B}^*=(\mu_0^*,\mu_1^*,\ldots,\mu_{n-1}^*)$.
It follows that $S_1^\perp=c_1^{-1}\langle \mu_{n-2}^*,\mu_{n-1}^* \rangle_{\fq}$ and by Lemma \ref{cor:dualbasis} we have that
\[ \mu_{n-2}^*=\delta^{-1}(a_{n-1}+b_1), \]
and
\[ \mu_{n-1}^*=\delta^{-1}, \]
where $f(x)=a_0+a_1x+\ldots+a_{n-1}x^{n-1}+x^n$ is the minimal polynomial of $b_1$ over $\fq$ and $\delta=f'(b_1)$.
Now, since $\mu_{n-2}^*,\mu_{n-1}^* \in (\lambda \overline{S})^\perp$ and since $(\lambda \overline{S})^\perp$ is a $1$-dimensional $\F_{q^{t}}$-subspace, it follows
\[ \frac{\mu_{n-2}^*}{\mu_{n-1}^*}=a_{n-1}+b_1 \in \F_{q^{t}}, \]
that is $b_1 \in \F_{q^{t}}$, a contradiction.
    \item[(i=1,j=3)] Since $S_1=c_1\langle 1,b_1 \ldots,b_1^{n-3} \rangle_{\F_{q}}$, then $S_1 \cap b_1 S_1=c_1\langle b_1, \ldots,b_1^{n-3} \rangle_{\F_{q}}$. Since $S_3$ is an $\F_{q^2}$-subspace then also $S_1$ has to be an $\F_{q^2}$-subspace and consequently $S_1 \cap b_1 S_1$ is an $\F_{q^2}$-subspace as well. Then $2$ divides $\dim_{\fq}(S_1 \cap b_1 S_1)=n-3$, this is a contradiction since $n$ is even. 
    \item[(i=2,j=3)] Since $S_2=\overline{S} \oplus c_2 \langle 1,b_2,\ldots,b_2^{t-3} \rangle_{\F_q}$, then $S_2 \cap b_2S_2= \overline{S} \oplus  c_2 \langle b_2,\ldots,b_2^{t-3} \rangle_{\F_q}$. Since $S_3$ is  an $\F_{q^2}$-subspace, then we can argue as in the previous case getting a contradiction to the fact that $n$ is even.
\end{itemize}
\end{proof}

Consequently to Lemma \ref{lem:sisjinequivalent}, we show that any two different $\fq$-subspaces obtained in Theorem \ref{cor:classn-2club} for different cases are $\Gamma\mathrm{L}(2,q^n)$-inequivalent.

\begin{theorem} \label{th:gammaln-2club}
Consider the following $\fq$-subspaces:
\begin{itemize}
    \item[(1)] Let 
    \[U_1=\left\{ \left(\sum_{i=0}^{n-3}c_1x_ib_1^i+\alpha+\beta a_1,\alpha +\beta b_1\right) \colon x_i,\alpha,\beta \in \F_{q} \right\},\] with $a_1,b_1,c_1 \in \F_{q^n}^*$ such that:
    \begin{itemize}
        \item $\F_q(b_1)=\F_{q^n}$,
        \item  $1 \in c_1\langle 1, b_1,\ldots,b_1^{n-3},b_1^{n-2}\rangle_{\F_{q}}$
        \item $a_1 \notin c_1\langle 1, b_1,\ldots,b_1^{n-3},b_1^{n-2}\rangle_{\F_{q}}$,
    \end{itemize}
    \item[(2)] if $n=t(\ell+1)$, with $\ell \in \mathbb{N}$, let
    \[U_2=\left\{ \left(\overline{s}+\sum_{i=0}^{t-3}c_2x_ib_2^i+\alpha+\beta a_2,\alpha +\beta b_2\ \right)  \colon x_i,\alpha,\beta \in \F_{q},\overline{s} \in \overline{S} \right\},\] with $a_2,b_2,c_2 \in \F_{q^n}^*$, $\overline{S} \subseteq \F_{q^n}$ such that:
    \begin{itemize}
        \item $\F_q(b_2)=\F_{q^t}$, with $3 \leq t \leq n-3$,
        \item $\overline{S}$
        $\F_{q^t}$-subspace of dimension $\ell$ such that $c_2\F_{q^t} \cap \overline{S}=\{0\}$,
        \item $1 \in \overline{S}\oplus c_2\langle 1,b_2,\ldots,b_2^{t-2} \rangle_{\F_q} $,
        \item $a \notin \overline{S}\oplus c_2\langle 1,b_2,\ldots,b_2^{t-2} \rangle_{\F_q} $,
    \end{itemize} 
    \item[(3)] if $n$ is even, let
    \[
    U_3=\left\{ \left(\sum_{i=0}^{n/2-2}c_3x_i\gamma^i+\alpha+\beta a_3,\alpha +\beta b_3\right)  \colon x_i \in \F_{q^2},\alpha,\beta \in \F_{q} \right\},
    \]
    with $a_3,b_3,c_3,\gamma \in \F_{q^n}^*$, such that: \begin{itemize}
        \item $\F_q(b_3)=\F_{q^2}$,
        \item $\F_{q^2}(\gamma)=\F_{q^n}$,
        \item $1 \in c_3\langle 1,\gamma,\ldots,\gamma^{n/2-2}\rangle_{\F_{q^2}}$,
        \item $a_3 \notin c_3\langle 1,\gamma,\ldots,\gamma^{n/2-2}\rangle_{\F_{q^2}}$.
    \end{itemize}
    Then two $\fq$-subspaces $U_i$ and $U_j$ (with $i\ne j$) are $\Gamma\mathrm{L}(2,q^n)$-inequivalent.
\end{itemize}
\end{theorem}
\begin{proof}
Suppose that $U_i$ and $U_j$ are $\Gamma\mathrm{L}(2,q^n)$-equivalent, then there exist a matrix $\left(\begin{array}{cc} A & B\\ C & D\end{array}\right) \in \mathrm{GL}(2,q^n)$ and $\rho \in \mathrm{Aut}(\fqn)$ such that
\begin{equation}\label{eq:equiv3cases} \left(\begin{array}{cc} A & B\\ C & D \end{array}\right) U_i^\rho =U_j. \end{equation}
Since $\langle(1,0)\rangle_{\fqn}$ is the only point in $L_{U_i}$ and $L_{U_j}$ having weight greater than 1 and the elements of $\mathrm{\Gamma L}(2,q^n)$ preserves the weight of points, then
\[ \left( \begin{matrix} 
A & B \\ C & D
\end{matrix} \right) \left( \begin{matrix} 
1 \\ 0
\end{matrix} \right)=\left( \begin{matrix} 
\rho \\ 0
\end{matrix} \right),
\]
for some nonzero $\rho \in \F_{q^n}$.
This implies that $C=0$, and therefore \eqref{eq:equiv3cases} implies $AS_i^{\rho}=S_j$, where $S_i$ and $S_j$ are as in Lemma \ref{lem:sisjinequivalent}. By applying Lemma \ref{lem:sisjinequivalent} we obtain that such $A$ cannot exist and hence $U_i$ and $U_j$ cannot be $\GammaL (2,q^n)$-equivalent.
\end{proof}

The above result allows us to prove the existence of new examples of $(n-2)$-clubs.

\begin{corollary}
If $6 \nmid n$, the constructions of Case (2) in Corollary \ref{cor:classn-2club} are new.
\end{corollary}
\begin{proof}
The previously known examples of $(n-2)$-clubs are those described in Constructions \ref{constr:(n-2)clubpolynomial} and \ref{constr:clubviascattered}. By Remark \ref{rk:constr1}, the Construction \ref{constr:(n-2)clubpolynomial} falls in Case (1) of Corollary \ref{cor:classn-2club} and the Constructions \ref{constr:clubviascattered}, by Remark \ref{vandervoorderm:n-2club}, give $(n-2)$-club only when either $2$ or $3$ divides $n$.
\end{proof}

\begin{remark}
Using Theorem \ref{th:gammaln-2club}, the $(n-2)$-clubs of Construction \ref{constr:clubviascattered} (cf. Remark \ref{vandervoorderm:n-2club}) are not $\GammaL(2,q^n)$-equivalent to the family of $(n-2)$-associated with $U_1$ in Theorem \ref{th:gammaln-2club} (that is Construction \ref{constr:(n-2)clubpolynomial}).
\end{remark}

\section{Polynomials defining clubs}\label{sec:polsoficlub}

In this section, we will provide a polynomial description for the known families of clubs and for those we have found in the previous section. As in the above section, in this section we will assume that $n\geq 5$ since if $n \leq 4$ we already have a polynomial description of clubs.

We start with the family of $(n-2)$-clubs of Construction \ref{constr:(n-2)clubpolynomial}.

\begin{theorem}\label{prop:policlub1}
Let $\lambda \in \fqn^*$ such that $\{1,\lambda,\ldots,\lambda^{n-1}\}$ is an $\fq$-basis of $\fqn$, the $\fq$-linear set $L_{\lambda}$ defined as in Construction \ref{constr:(n-2)clubpolynomial} is $\mathrm{PG}\mathrm{L}(2,q^n)$-equivalent to $L_{U_p}$ with $U_p=\{(x,p(x) \colon x \in \F_{q^n}\}$ and 
\[ p(x)=\mathrm{Tr}_{q^n/q}(c_{n-2}x)+\lambda\mathrm{Tr}_{q^n/q}(c_{n-1}x) \]
where $\omega \in \fqn$ is such that $(1,\lambda,\ldots,\lambda^{n-3},\lambda^{n-2}+\omega,\omega\lambda)$ is an order $\fq$-basis of $\fqn$ and $(c_0,\ldots,c_{n-1})$ is its dual basis.
In particular, if $q$ is odd then we may choose $\omega= \lambda^{n-2}$ and in this case
\[c_i=\frac{1}{\delta}\sum_{j=0}^{n-i-1} \lambda^j a_{i+j+1}, \]
for $i \in \{n-2,n-1\}$, where $f(x)=a_0+a_1x+\ldots+a_{n-1}x^{n-1}+x^n$ is the minimal polynomial of $\lambda$ over $\fq$ and $\delta=f'(\lambda)$.
\end{theorem}
\begin{proof}
Consider
\[U_{\lambda}=\{ (t_1\lambda+\ldots+t_{n-1}\lambda^{n-1},t_{n-1}+t_n \lambda) \colon (t_1,\ldots,t_n)\in \fq^n \}.\]
Note that such $\omega$ as in the statement there exists since the set
\[ Y= \left\{ \frac{-\alpha_0-\ldots-\alpha_{n-2}\lambda^{n-2}}{\alpha_{n-2}+\lambda \alpha_{n-1}} \colon \alpha_0,\ldots,\alpha_{n-1} \in \fq, (\alpha_{n-2},\alpha_{n-1})\ne (0,0) \right\} \]
does not cover all the elements of $\fqn$. Indeed, $\lvert Y \rvert = \lvert L_{U_{\lambda}}\rvert -1=q^{n-1}+q^{n-2}$.
Now, applying $\left( \begin{array}{cc} \lambda^{-1} & \omega \\ 0 & 1 \end{array} \right)$ to $U_{\lambda}$ we obtain the following subspace
\[ U=\{ (t_1+\ldots+t_{n-1}\lambda^{n-2}+\omega(t_{n-1}+t_n \lambda),t_{n-1}+t_n \lambda) \}. \]
Since $(c_0,\ldots,c_{n-1})$ is the dual basis of $(1,\lambda,\ldots,\lambda^{n-3},\lambda^{n-2}+\omega,\omega\lambda)$, then
\[ p(t_1+\ldots+t_{n-1}\lambda^{n-2}+\omega(t_{n-1}+t_n \lambda))=t_{n-1}+\lambda t_n, \]
that is $U=U_p$ and the first part of the assertion is proved.
The second part follows from the fact that when $q$ is odd we may choose $\omega=\lambda^{n-2}$ so that $U$ can be written as follows
\[\{ (t_1+\ldots+2t_{n-1}\lambda^{n-2}+t_n \lambda^{n-1},t_{n-1}+t_n \lambda) \colon (t_1,\ldots,t_n)\in \fq^n \}.\]
Since now the first component can be expressed through a polynomial basis the assertion follows by Corollary \ref{cor:dualbasis}.
\end{proof}

\begin{remark}
In the case in which the minimal polynomial of $\lambda$ has degree two or three we may use Corollary \ref{cor:binetrin} to obtain nicer expressions for $p(x)$.
\end{remark}

Regarding Construction \ref{constr:clubviascattered} we can give the relative polynomial description which will rely on the associated scattered polynomial (for $q=2$ see also \cite{DeBoeckVdV19}).

\begin{theorem}\label{th:polscattclub}
Let $n=tr$, $t,r>1$, $f \in \mathcal{L}_{t,q}$ a scattered polynomial. 
Let 
\[U_{a,b}=\left\{ \left(f(x_0)-ax_0,bx_0+\sum_{i=1}^{r-1}x_i \omega^i\right) \colon x_i \in \F_{q^t} \right\},\]
for some fixed $a,b \in \F_{q^t}$ with $b\ne 0$ and $\{1,\omega,\ldots,\omega^{r-1}\}$ an $\F_{q^t}$-basis of $\F_{q^n}$. Then $L_{U_{a,b}}$ is $\mathrm{PG}\mathrm{L}(2,q^n)$-equivalent to $L_{U_p}$ with $U_p=\{(x,p(x) \colon x \in \F_{q^n}\}$ and 
\[ p(x)=\mathrm{Tr}_{q^n/q^t}(f(x)-x).\]
\end{theorem}
\begin{proof}
First observe that $U_{a,b}$ is $\mathrm{GL}(2,q^n)$-equivalent to
\[ U=\left\{ \left(f(x_0)-ax_0,x_0+\sum_{i=1}^{r-1}x_i \omega^i\right) \colon x_i \in \F_{q^t} \right\}.\]
Then consider $(c_0,\ldots,c_{t-1})$ be the dual basis of $(1,\omega,\ldots,\omega^{t-1})$.
Noting that if $x=\sum_{i=0}^{t-1}x_i \omega^i$ then $x_0=\mathrm{Tr}_{q^n/q^r}(c_0x)$ and applying $\left( \begin{array}{cc} 0 & 1 \\ 1 & 0 \end{array} \right)$ we obtain that $U_{a,b}$ is $\mathrm{GL}(2,q^n)$-equivalent to 
\[ U'=\{ (x, f(\mathrm{Tr}_{q^n/q^t}(c_0x))-a\mathrm{Tr}_{q^n/q^t}(c_0x)) \colon x \in \F_{q^t} \}, \]
since $f \circ \mathrm{Tr}_{q^n/q^t} = \mathrm{Tr}_{q^n/q^t} \circ f$, and by applying $\left( \begin{array}{cc} c_0 & 0 \\ 0 & 1 \end{array} \right)$ to $U'$ we obtain that $U'$ is $\mathrm{GL}(2,q^n)$-equivalent to $U_p$ and hence the assertion. 
\end{proof}

Now, we give the polynomial description of any possible $(n-2)$-clubs.

\begin{theorem}\label{th:(n-2)clubpoldescript}
Let $L_W$ be an $(n-2)$-club in $\PG(1,q^n)$.
There exist $b \in \fqn\setminus\fq$ and $\xi,\eta \in \fqn^*$ such that $L_W$ is $\mathrm{PGL}(2,q^n)$-equivalent to $L_p$, where
\[ p(x)=\mathrm{Tr}_{q^n/q}(\xi x)+b \mathrm{Tr}_{q^n/q}(\eta x), \]
with $\langle \xi,\eta \rangle_{\fq}^\perp=\overline{S}\oplus c \langle 1,b,\ldots,b^{t-2}\rangle_{\fq}$, where $t=[\fq(b),\fq]$, $c \in \fqn^*$ and $\overline{S}$ an $\F_{q^t}$-subspace of $\fqn$ of dimension $n/t-1$.
\end{theorem}
\begin{proof}
Let $L_W$ be an $(n-2)$-club, then by Corollary \ref{cor:classn-2club} $L_W$ is $\mathrm{PGL}(2,q^n)$-equivalent to $L_{U}$, where $n=t(\ell+1)$, with $\ell \in \mathbb{N}_0$, and
\[U=\left\{ \left(\overline{s}+\sum_{i=0}^{t-3}cx_ib^i+\alpha+\beta a,\alpha +\beta b \right)  \colon x_i,\alpha,\beta \in \F_{q},\overline{s} \in \overline{S} \right\},\] 
with $a,b,c \in \F_{q^n}^*$, $\overline{S} \subseteq \F_{q^n}$  such that
\begin{itemize}
       \item $\F_q(b)=\F_{q^t}$,
        \item $\overline{S}$ is an
       $\F_{q^t}$-subspace of dimension $\ell$ such that $c\F_{q^t} \cap \overline{S}=\{0\}$,
       \item $1 \in \overline{S} \oplus c\langle 1,b,\ldots,b^{t-2} \rangle_{\F_q}$,
        \item $a \notin \overline{S}\oplus c\langle 1,b,\ldots,b^{t-2} \rangle_{\F_q} $.
\end{itemize} 
Now, we prove the existence of $\omega \in \fqn^*$ such that $(s_1,\ldots,s_{\ell t},c,cb,\ldots,cb^{t-3},1+\omega, a+\omega b)$ is an ordered $\fq$-basis, where $\{s_1,\ldots,s_{\ell t}\}$ is an $\fq$-basis of $\overline{S}$.
Indeed, $\omega$ exists if and only if the set
\[ \left\{ \frac{d_1 s_1+ d_2s_2+\ldots+d_{n-2} c b^{n-3}+d_{n-1}+ad_{n}}{d_{n-1}+d_{n}b} \colon d_1,\ldots,d_{n}\in \fq, (d_{n-1},d_{n})\ne (0,0)  \right\} \]
does not cover all the elements of $\F_{q^n}$ and this is equivalent to the fact that there exists a point of the form $\langle (1,\omega) \rangle_{\fqn}$ which is not in $L_{U}$.
This is always true since $L_{U}$ has size $q^{n-1}+q^{n-2}+1$ which is always less than $q^n-1$.
Finally, applying $\left( \begin{array}{cc} 1 & \omega \\ 0 & 1 \end{array} \right)$ to $U$ we obtain $U_p$, where
$p(x)=\mathrm{Tr}_{q^n/q}(\xi_{n-2}x)+b \mathrm{Tr}_{q^n/q}(\xi_{n-1}x)$ and $(\xi_0,\ldots,\xi_{n-1})$ is the dual basis of $(s_1,\ldots,s_{\ell t},c,cb,\ldots,cb^{t-3},1+\omega, a+\omega b)$. This is due to the fact that every element $\delta$ of $\fqn$ may be written as $\delta=\overline{s}+\sum_{i=0}^{t-3}cx_ib^i+\alpha+\beta a$ for some $\overline{s} \in \overline{S}$, $x_0,\ldots,x_{t-3},\alpha, \beta \in \fq$ and
$ p( \delta)= \alpha + b \beta$. Therefore $\xi$ and $\eta$ of the statement correspond to $\xi_{n-2}$ and to $\xi_{n-1}$.
\end{proof}

\section{Blocking sets and KM-arcs}\label{sec:blocksetsKMarcs}

A well-studied topic in finite geometry is the theory of blocking set.
A \emph{blocking set} $\mathcal{B}$ in $\PG(2,q)$ is a set of points with the property that every line meets $\mathcal{B}$ in at least one point. 
An easy example is given by a line, which is also called a \emph{trivial blocking set}.
An important class of blocking sets is given by the blocking sets of R\'edei type. If $\mathcal{B}$ is a blocking set in $\PG(2,q)$ of size $q+t$ with $t<q$ for which there exists a line $\ell$ meeting $\mathcal{B}$ in exactly $t$ points, then $\mathcal{B}$ is called a blocking set of \emph{R\'edei type}. Related to the celebrated linearity conjecture are the linear blocking sets.
An \emph{$\fq$-linear blocking set} $L_W$ in $\PG(2,q^n)$ is any $\fq$-linear set of rank $n+1$, and those of R\'edei type are the one that can be described as follows: let $\ell$ be the R\'edei line of $\mathcal{B}$ and consider $L_U = \ell \cap L_W$ and $v \in W \notin \langle U \rangle_{\F_{q^n}}$. 
Then $L_W=L_{U\oplus \langle v \rangle_{\fq}}$ and $L_U$ is an $\fq$-linear set of rank $n$ contained in the line $\ell$.
The R\'edei line $\ell$ is unique if $L_U$ is a strictly $\fq$-linear set which is not an $(n-1)$-club, equivalently $L_U$ is not $\mathrm{P\Gamma L}(2,q^n)$-equivalent to $L_{\mathrm{Tr}_{q^n/q}}$.

We are now interested in studying those associated with $i$-clubs.

\begin{remark}\label{rem:Redeiblocksets}
Let $\ell_{\infty}$ be the line in $\PG(2,q^n)=\PG(V,\F_{q^n})$ with equation $X_2=0$ and consider $L_U \subseteq \ell_{\infty}$ to be an $i$-club with $i\leq n-2$.
Consider $v \in V\setminus \langle U \rangle_{\F_{q^n}}$ and $W=U \oplus \langle v \rangle_{\fq}$. Then $L_W$ is an $\fq$-linear blocking set of R\'edei type in $\PG(2,q^n)$ with size $q^n+q^{n-1}+\ldots+q^i+1$ in which all but one of the points have weight one and the remaining one has weight $i$.
Moreover, for any line $\ell$ then
\[ w_{L_U}(\ell) \in \{1, 2, i+1, n \}. \]
In particular, since $i<n-1$, there exists only one line having weight $n$ and exactly $q^{n-i}$ lines having weight $i+1$.
Moreover, $L_W$ has a $(q+1)$-secant line since $L_U$ contains at least one point of weight one.
\end{remark}

Using the classification proved for $(n-2)$-clubs we obtain the classification of linear blocking sets of small R\'edei type with some constraints on the weight of the points.
Recall that a blocking set in $\PG(2,q^n)$ is said to be \emph{small} if $|\mathcal{B}|\leq 3(q^n+1)/2$.

\begin{theorem}\label{th:classRedeibs}
Suppose that $\mathrm{char}(\fqn)>2$ and $n\geq 5$.
Let $\mathcal{B}$ be a small minimal blocking set in $\PG(2,q^n)$ of R\'edei type and suppose that for every line $s$ different from the R\'edei line $\ell_{\infty}$ we have
\[ |s \cap \mathcal{B}| \in \{ 1,q+1,q^{n-2}+1 \},  \]
and for all of the values there exists at least one line meeting $\mathcal{B}$ in this number of points.

Then $\mathcal{B}$ is $\mathrm{PGL}(3,q^n)$-equivalent to $L_{\overline{U}}$ where $\overline{U}$ has the following form: let $n=t(\ell+1)$, with $\ell \in \mathbb{N}_0$ and
\[\overline{U}=\left\{ \left(\overline{s}+\sum_{i=0}^{t-3}cx_ib^i+\alpha+\beta a,\alpha +\beta b, \delta \right)  \colon x_i,\alpha,\beta, \delta \in \F_{q},\overline{s} \in \overline{S} \right\},\] with $a,b,c \in \F_{q^n}^*$, $\overline{S} \subseteq \F_{q^n}$ such that
\begin{itemize}
    \item $\F_q(b)=\F_{q^t}$,
    \item $\overline{S}$ is an $\F_{q^t}$-subspace of dimension $\ell$ such that $c\F_{q^t} \cap \overline{S}=\{0\}$,
    \item $1 \in \overline{S}\oplus c\langle 1,b,\ldots,b^{t-2} \rangle_{\F_q}$,
    \item $a \notin \overline{S}\oplus c\langle 1,b,\ldots,b^{t-2} \rangle_{\F_q}$.
\end{itemize} 
\end{theorem}
\begin{proof}
Since $\mathcal{B}$ is a small minimal blocking set in $\PG(2,q^n)$ of R\'edei type with a $(q+1)$-secant line, then $\mathcal{B}$ is an $\fq$-linear set, that is $\mathcal{B}=L_W$, see \cite[Theorem 4.3 (iv)]{surveybloset} (which relies on the well-celebrated results in \cite{Ball,BBBSSz}).
Since $L_W$ is an $\fq$-linear blocking set of R\'edei type, then we have $W=U\oplus \langle v \rangle_{\fq}$ and $L_U$ is an $\fq$-linear set of rank $n$ contained in the R\'edei line $\ell_{\infty}$. 
Since every line different from the R\'edei line has weight $1$, $2$ or $n-1$ in $L_W$, it follows that the points of $L_U$ have weight either $1$ or $n-2$ and $n\geq 5$, it results that $L_U$ is an $(n-2)$-club and the assertion now follows by Corollary \ref{cor:classn-2club}.
\end{proof}

\begin{remark}
Since \cite[Theorem 4.3 (iv)]{surveybloset} needs the assumption $\mathrm{char}(\fqn)>2$, in the above statement we have to add such a condition. However, if we add the assumption on $\mathcal{B}$ to be linear, then the above result holds for every characteristic.
\end{remark}

Moreover, from the Construction \ref{constr:clubviascattered}, we have examples of linear blocking sets of R\'edei type of the following form.

\begin{remark}\label{prop:constbsfromclubs}
Let $n=rt$, $t,r>1$, $f \in \mathcal{L}_{t,q}$ a scattered polynomial. 
Let 
\[U_{a,b}=\left\{ \left(f(x_0)-ax_0,bx_0+\sum_{i=1}^{r-1}x_i \omega^i,0\right) \colon x_i \in \F_{q^t} \right\},\]
for some fixed $a,b \in \F_{q^t}$ with $b\ne 0$ and $\{1,\omega,\ldots,\omega^{r-1}\}$ an $\F_{q^t}$-basis of $\F_{q^n}$.
Let 
\[ i=\left\{ \begin{array}{ll} t(r-1), & \text{if}\,\, f(x)-ax\,\, \text{is invertible over}\,\,\F_{q^t},\\
t(r-1)+1, & \text{otherwise}.\end{array} \right. \]
Then
\[ W=U_{a,b}\oplus \langle (0,0,1) \rangle_{\fq} \]
defines an $\fq$-linear blocking set of R\'edei type with size $q^{n}+q^{n-1}+\ldots+q^i+1$.
\end{remark}

We now prove that in the family of blocking set described in the above proposition, there are several inequivalent examples.
To this aim we recall the following result.

\begin{lemma}\cite[Proposition 2.3]{BoPol}\label{lem:equivBoPol}
Let $L_W$ be an $\fq$-linear blocking set with a $(q+1)$-secant line. Then $L_W=L_{W'}$ if and only if there exists $\lambda \in \fqn^*$ such that $W'=\lambda W$.
\end{lemma}

Let $L_W$ and $L_{W'}$ be two $\fq$-linear sets of R\'edei type with R\'edei lines $\ell_1$ and $\ell_2$, respectively, and let $\ell_1\cap L_W=L_{U}$ and $\ell_2\cap L_{W'}=L_{U'}$. Then by the above lemma $L_{W}$ and $L_{W'}$ are $\mathrm{P\Gamma L}(3,q^n)$-equivalent if and only if $W$ and $W'$ are $\mathrm{\Gamma L}(3,q^n)$-equivalent. Moreover, if $L_W$ and $L_{W'}$ has only one R\'edei line, then they are $\mathrm{P\Gamma L}(3,q^n)$-equivalent if and only if $U$ and $U'$ are $\mathrm{\Gamma L}(2,q^n)$-equivalent.
Therefore, from Theorem \ref{th:equivscattclub} we have the following result.

\begin{corollary}
The number of $\mathrm{P\Gamma L}(3,q^n)$-inequivalent blocking sets of the family described in Remark \ref{prop:constbsfromclubs} coincide with the number of $\mathrm{\Gamma L}(2,q^t)$-inequivalent scattered $\fq$-subspaces in $\F_{q^t}^2$ of dimension $t$.
\end{corollary}

KM-arcs have been originally introduced in \cite{KM} by Korchmar\'os and Mazzocca.
A \emph{KM-arc} $\mathcal{A}$ of type $t$ is a set of $q+t$ points in $\PG(2,q)$ such that every line contains $0,2$ or $t$ points. When $2<t<q$, in \cite[Theorem 2.5]{KM} it has been proved that if KM-arc of type $t$ exists then $q$ is even and $t \mid q$.
Moreover, $\mathcal{A}$ is called a \emph{translation} KM-arc if there exists a line $\ell$ of $\PG(2,q)$ such that the group of elations with axis $\ell$ and fixing $\mathcal{A}$ acts transitively on the points of $\mathcal{A}\setminus \ell$. 
In \cite{DeBoeckVdV2016}, De Boeck and Van de Voorde proved several and important results on translation KM-arcs and pointed out that the main question on studying this object regards for which values of $q$ and $t$ a KM-arc of type $t$ exists in $\PG(2,q)$ and which nonequivalent one there exist. 
Using the link between translation KM-arcs and $i$-clubs established in \cite[Theorem 2.2]{DeBoeckVdV2016} we will give a classification for translation KM-arcs of type $2^{n-2}$ in $\PG(2,2^n)$. Although this result has been already established in \cite[Theorem 4.12]{DeBoeckVdV2016}, our result give a more direct expression of these sets. 
To this aim let's recall \cite[Theorem 2.2]{DeBoeckVdV2016}. 

\begin{construction}\label{constr:KMarcs}
Let $\ell \colon X_2=0 \subset \PG(2,2^n)=\PG(V,\F_{2^n})$ and consider $L_U \subset \ell$ to be an $i$-club. 
Let $v \in V \notin \langle U \rangle_{\F_{2^n}}$ and let $U'=U \oplus \langle v \rangle_{\F_2}$. 
Consider the following set of points
\[ \mathcal{A}(U,v)=(L_{U'} \setminus \ell) \cup (\ell \setminus L_U). \]
In \cite[Theorem 2.1]{DeBoeckVdV2016} it has been proved that $\mathcal{A}(U,v)$ is a translation KM-arc of type $2^i$ in $\PG(2,2^n)$.
\end{construction}

Actually, all the translation KM-arcs can be constructed as before.

\begin{theorem}\cite[Theorem 2.2]{DeBoeckVdV2016}\label{th:KMarcfromclub}
Every translation KM-arc of type $2^i$ in $\PG(2,2^n)$ can be constructed as in Construction \ref{constr:KMarcs} by using an $i$-club.
\end{theorem}

The first example of KM-arc of type $2^{m(\ell-1)}$ in $\PG(2,2^{\ell m})$ was presented in \cite{KM} and can be described via the club $L_{T}$, where $T(x)$ is as in Equation \ref{eq:firstKMarc}; see \cite[Theorem 3.2]{DeBoeckVdV2016}.
The clubs arising from Construction \ref{constr:clubviascattered} define the KM-arcs found by G\'acs and Weiner in \cite{GW2003}.

We are now able to prove the following classification result for translation KM-arcs of type $2^{n-2}$ in $\PG(2,2^n)$. These KM-arcs have been already classified in \cite[Theorem 4.12]{DeBoeckVdV2016}, but our result give a more explicit description.

\begin{theorem}
Let $\mathcal{A}$ be a translation KM-arcs of type $2^{n-2}$ in $\PG(2,2^n)$. Then $\mathcal{A}$ is $\mathrm{PGL}(2,q^n)$-equivalent to $\mathcal{A}(U,(0,0,1))$, where $U$ is described as in Corollary \ref{cor:classn-2club}.
\end{theorem}
\begin{proof}
Since $\mathcal{A}$ be a translation KM-arcs of type $2^{n-2}$ in $\PG(2,2^n)$, by Theorem \ref{th:KMarcfromclub}  $\mathcal{A}$ is $\mathrm{PGL}(2,2^n)$-equivalent to $\mathcal{A}(U,(0,0,1))$ where $L_U$ is an $(n-2)$-club. Using the classification of $(n-2)$-clubs provided in Corollary \ref{cor:classn-2club} we obtain the assertion.
\end{proof}

\section{Linearized polynomials with conditions on its value set}

The problem of estimating the size of the value set of a polynomial over finite field or finding polynomials with large/small value set is a quite classical problem in the theory of polynomials over finite field; see \cite{valuesets}. When considering linearized polynomials, we have more information when considering its quotient with $x$. More precisely, let $f(x)$ be a linearized polynomial, then the value set we are interested in is
\[ \mathcal{V}\left(\frac{f(x)}x\right)=\left\{ \frac{f(\alpha)}{\alpha}  \colon \alpha \in \F_{q^n}^*\right\}. \]
If $f(x)$ is a $q$-polynomial having $\F_q$ as maximum field of linearity, i.e. $f(x)$ does not define an $\F_{q^i}$-linear map from $\F_{q^n}$ to $\F_{q^n}$ for $1<i<n$ and $i \mid n$, by the results in \cite{Ball,BBBSSz}, we have the following bounds
\begin{equation}\label{eq:size}
q^{n-1}+1 \leq \left|\mathcal{V}\left(\frac{f(x)}x\right)\right|\leq \frac{q^n-1}{q-1}.
\end{equation}
This is indeed connected with the directions of $f$.
Consider the projective plane $\mathrm{PG}(2, q^n)$ as the projective closure of  $\mathrm{AG}(2, q^n)$ via the line $\ell_\infty(= \mathrm{PG}(1, q^n))$, that without loss of generality we may assume to have equation $X_2=0$.
Let $f(x)$ be any linearized polynomial in $\mathcal{L}_{n,q}$, then
\[ \mathcal{V}\left(\frac{f(x)}x\right)=\mathcal{D}_f=\left\{\frac{f(x)-f(y)}{x-y} \colon x,y \in \fqn, x \ne y \right\} \]
is the set $\mathcal{D}_f$ of directions determined by the graph of $f$, that is the set $\mathcal{G}_f=\{\langle (x,f(x),1)\rangle_{\fqn} \colon x \in \fqn\}\subseteq \mathrm{AG}(2,q^n)$. 
In the case in which $f(x)$ is a $q$-polynomial having $\F_q$ as maximum field of linearity, then by \eqref{eq:size} we have
\[
q^{n-1}+1 \leq \left|\mathcal{D}_f\right|\leq \frac{q^n-1}{q-1}.
\]
If $\mathcal{D}_f$ is a scattered $\fq$-linear set in $\mathrm{PG}(1,q^n)$, then $\left|\mathcal{V}\left(\frac{f(x)}x\right)\right|= \frac{q^n-1}{q-1}$, whereas if $f(x)=\mathrm{Tr}_{q^n/q}(x)$ or $f(x)$ is as in \cite{NPSZ2022minsize} then $\left|\mathcal{V}\left(\frac{f(x)}x\right)\right|= q^{n-1}+1$. So, both the bounds in \eqref{eq:size} are sharp.

Two polynomials $f$ and $g$ are said to be \emph{equivalent} if there exists $\varphi \in \mathrm{\Gamma L}(3,q^n)$ such that $\varphi (\mathcal{G}_f)=\mathcal{G}_g$, see \cite{CsMPq5}.
Using this definition of equivalence, in \cite[Problem 4]{CsMPq5} the authors asked to classify and to find more examples, up to equivalence, of $q$-polynomials $f(x)\in \mathcal{L}_{n,q}$ such that in the multiset $\{f(\alpha)/\alpha \colon \alpha \in \fqn^* \}$ there is a unique element which is represented more than $q-1$ times, namely $q^i-1$ times for some $i \in \{1,\ldots,n-1\}$. We call these polynomials \emph{$i$-club polynomials} because they correspond to those polynomials for which $L_f$ is an $i$-club.

Up to our knowledge, only two examples of $i$-clubs polynomials were previously known: the one presented in \eqref{eq:firstKMarc} and some examples when $n\leq 5$ (see \cite{BoPol,CsZ2018,BMZZ}).

We will first show the new examples of club polynomials.

\begin{proposition}
Let $\lambda \in \fqn^*$ such that $\{1,\lambda,\ldots,\lambda^{n-1}\}$ is an $\fq$-basis of $\fqn$, $\omega \in \fqn$ is such that $(1,\lambda,\ldots,\lambda^{n-3},\lambda^{n-2}+\omega,\omega\lambda)$ is an order $\fq$-basis of $\fqn$ and $(b_0,\ldots,b_{n-1})$ is its dual basis.
Then
\[ p(x)=\mathrm{Tr}_{q^n/q}(b_{n-2}x)+\lambda\mathrm{Tr}_{q^n/q}(b_{n-1}x) \]
is an $(n-2)$-club polynomial.
Let $n=rt$, $t,r>1$, $f\in \mathcal{L}_{t,q}$ be a scattered $q$-polynomial.
Then 
\[ p(x)=\mathrm{Tr}_{q^n/q^t}(f(x)-x),\]
is an $i$-club polynomial with 
\[ i=\left\{ \begin{array}{ll} t(r-1), & \text{if}\,\, f(x)-x\,\, \text{is invertible over}\,\,\F_{q^t},\\
    t(r-1)+1, & \text{otherwise}.\end{array} \right. \]
\end{proposition}
\begin{proof}
It follows from the fact that if $L_f$ is an $i$-club, then $f(x)$ is an $i$-club polynomial.
\end{proof}

We are now able to give a classification of $(n-2)$-club polynomials.

\begin{theorem}
Let $f\in \mathcal{L}_{n,q}$ an $(n-2)$-club polynomial. Then $f$ is equivalent to $p(x)=\mathrm{Tr}_{q^n/q}(\xi x)+b\mathrm{Tr}_{q^n/q}(\eta x)$, where $b,\xi$ and $\eta$ are as in Theorem \ref{th:(n-2)clubpoldescript}.
\end{theorem}
\begin{proof}
The proof again follows from the fact that $f$ is an $(n-2)$-club if and only if $U_f$ defines an $(n-2)$-club. This can happen if and only if $U_f$ is $\mathrm{GL}(2,q^n)$-equivalent to $U_g$, where $g$ is one of the polynomials described in Theorem \ref{th:(n-2)clubpoldescript}.
\end{proof}

\section{Linear rank metric codes}\label{sec:rankmetric}

Rank metric codes were introduced by Delsarte \cite{Delsarte} in 1978 as subsets of matrices and they have been intensively investigated in recent years because of their applications; we refer to \cite{sheekey_newest_preprint,PZ}.
In this section we will be interested in rank metric codes in $\F_{q^m}^n$.
The \emph{rank} (weight) $w(v)$ of a vector $v=(v_1,\ldots,v_n) \in \F_{q^m}^n$ is defined as the dimension of the vector space generated over $\F_q$ by its entries, i.e $w(v)=\dim_{\fq} (\langle v_1,\ldots, v_n\rangle_{\fq})$. 

A \emph{(linear vector) rank metric code} $\C $ is an $\F_{q^m}$-subspace of $\F_{q^m}^n$ endowed with the rank distance, where such a distance is defined as
\[
d(x,y)=w(x-y),
\]
where $x, y \in \F_{q^m}^n$. 

Let $\C \subseteq \F_{q^m}^n$ be a linear rank metric code. We will write that $\C$ is an $[n,k,d]_{q^m/q}$ code (or $[n,k]_{q^m/q}$ code) if $k=\dim_{\F_{q^m}}(\C)$ and $d$ is its minimum distance, that is 
\[
d=\min\{d(x,y) \colon x, y \in \C, x \neq y  \}.
\]

By the classification of $\F_{q^m}$-linear isometry of $\F_{q^m}^n$ (see \cite{berger2003isometries}), we say that two rank metric codes $\C,\C' \subseteq \F_{q^m}^n$ are \emph{(linearly) equivalent} if and only if there exists a matrix $A \in \mathrm{GL}(n,q)$ such that
$\C'=\C A=\{vA : v \in \C\}$. 
Moreover, we will say that $\C$ and $\C'$ are \emph{semilinearly equivalent} if and only if there exist a matrix $A \in \mathrm{GL}(n,q)$ and an automorphism $\rho \in \mathrm{Aut}(\F_{q^m})$ such that
$\C'=\C^{\rho} A=\{v^{\rho}A : v \in \C\}$,  where the action of $\rho$ is extended entrywise. 
Most of the codes we will consider are \emph{non-degenerate}, i.e. those for which the columns of any generator matrix of $\C$ are $\fq$-linearly independent. Denote by $ \mathfrak{C}[n,k,d]_{q^m/q}$ the set of all linear $[n,k,d]_{q^m/q}$ rank metric codes in $\F_{q^m}^n$.

The geometric counterpart of rank metric are the systems. 
An $[n,k,d]_{q^m/q}$ \emph{system} $U$ is an $\F_q$-subspace of $\F_{q^m}^k$ of dimension $n$, such that
$ \langle U \rangle_{\F_{q^m}}=\F_{q^m}^k$ and
$$ d=n-\max\left\{\dim_{\F_q}(U\cap H) \mid H \textnormal{ is an $\F_{q^m}$-hyperplane of }\F_{q^m}^k\right\}.$$
Moreover, two $[n,k,d]_{q^m/q}$ systems $U$ and $U'$ are \emph{equivalent} if there exists an $\F_{q^m}$-isomorphism $\varphi\in\mathrm{GL}(k,\F_{q^m})$ such that
$$ \varphi(U) = U'.$$
We denote the set of equivalence classes of $[n,k,d]_{q^m/q}$ systems by $\mathfrak{U}[n,k,d]_{q^m/q}$.

\begin{theorem}(see \cite{Randrianarisoa2020ageometric}) \label{th:connection}
Let $\C$ be a non-degenerate $[n,k,d]_{q^m/q}$ rank metric code and let $G$ be an its generator matrix.
Let $U \subseteq \F_{q^m}^k$ be the $\F_q$-span of the columns of $G$.
The rank weight of an element $x G \in \C$, with $x \in \F_{q^m}^k$ is
\begin{equation}\label{eq:relweight}
w(x G) = n - \dim_{\fq}(U \cap x^{\perp}),\end{equation}
where $x^{\perp}=\{y \in \F_{q^m}^k \colon x \cdot y=0\}.$ In particular,
\begin{equation} \label{eq:distancedesign}
d=n - \max\left\{ \dim_{\fq}(U \cap H)  \colon H\mbox{ is an } \F_{q^m}\mbox{-hyperplane of }\F_{q^m}^k  \right\}.
\end{equation}
\end{theorem}

Actually, the above result allows us to give a one-to-one correspondence between equivalence classes of non-degenerate $[n,k,d]_{q^m/q}$ codes and equivalence classes of $[n,k,d]_{q^m/q}$ systems, see \cite{Randrianarisoa2020ageometric}.
The system $U$ and the code $\C$ and in Theorem \ref{th:connection} are said to be \emph{associated}.

Moreover, the semilinear inequivalence on linear rank metric codes can be read also on the associated systems via the action of $\mathrm{\Gamma L}(k,q^m)$ on the $\fq$-subspaces of $\F_{q^m}^n$.

\begin{theorem}(see \cite{John} and \cite{JohnGeertrui}) \label{th:semilinearequivalence}
Let $\C$ and $\C'$ two linear $[n,k,d]_{q^m/q}$ non-degenerate rank metric codes and let $U$ and $U'$ be two associated systems with $\C$ and $\C'$, respectively.
Then $\C$ and $\C'$ are semilinearly equivalent if and only if $U$ and $U'$ are $\mathrm{\Gamma L}(k,q^m)$-equivalent.
\end{theorem}

\subsection{$2$-dimensional linear rank metric codes}

Let $U$ be an $\fq$-subspace of $\F_{q^n}^2$ of dimension $n$ such that $\langle U \rangle_{\F_{q^n}}=\F_{q^n}^2$. Then $U$ is an $[n,2,d]_{q^n/q}$ system where $d=n-\max \{ w_{L_U}(P) \colon P \in \mathrm{PG}(1,q^n) \}$.
Up to the action of $\mathrm{GL}(2,q^n)$, as already said in Section \ref{sec:linset}, we can assume that 
\[ U=\{(x,f(x)) \colon x \in \mathbb{F}_{q^n}\}, \]
for some $f \in \mathcal{L}_{n,q}$. Then a code associated with $U$ is the code $\mathcal{C}_f$ whose generator matrix is
\[ G=\left(
\begin{array}{cccc}
    \xi_1 &   \xi_2 & \cdots & \xi_n \\
    f(\xi_1) & f(\xi_2) & \cdots & f(\xi_n)
\end{array}
\right),
\]
where $\xi_1,\ldots,\xi_n$ is an $\fq$-basis of $\fqn$.
This means that a code associated with $U$ is given by the evaluation of the polynomials in $\langle x, f(x) \rangle_{\fqn}$.

Let determine the weight distribution of the codes associated with $i$-clubs.

\begin{proposition}\label{prop:codefromiclub}
Let $U$ be an $\fq$-subspace of $\F_{q^n}^2$ of dimension $n$ such that $\langle U \rangle_{\F_{q^n}}=\F_{q^n}^2$.
If $L_U$ is an $i$-club, then the weight distribution of an associated code is the following: 
\begin{itemize}
    \item $q^n-1$ codewords of weight $n-i$;
    \item $(q^n-1)(q^{n-1}+\ldots+q^i)$ of weight $n-1$;
    \item $(q^n-1)(q^n-q^{n-1}-\ldots-q^i)$ of weight $n$.
\end{itemize}  
Conversely, let $\C$ be a linear $[n,2,n-i]_{q^n/q}$ non-degenerate rank metric code with $q^n-1$ codewords of weight $n-i$ and all the remaining nonzero codewords having weight either $n-1$ or $n$, then any associate system with $\mathcal{C}$ defines an $i$-club.
\end{proposition}
\begin{proof}
The proof follows by Theorem \ref{th:connection} and from the fact that an $i$-club has size $q^{n-1}+\ldots+q^i+1$ points, one point of weight $i$ and all the remaining points have weight one.
\end{proof}

Maximum rank distance codes of dimension $2$ are exactly those that have only codewords of weight $n-1$ and $n$. However, the codes associated with $i$-clubs are very \emph{close} to them and for this reason they are certainly of interest. Also, they are codes for which the nonzero weights are only three, see \cite{PSSZ202x}.

Therefore, we call the examples of rank metric codes as in Proposition \ref{prop:codefromiclub} \emph{$i$-club rank metric codes}. 
We list now the examples of $i$-club rank metric codes using the polynomial description we found in Section \ref{sec:polsoficlub}.

\smallskip

\textbf{Examples of $(n-1)$-club rank metric codes.}\\
Up to equivalence, they are all of the form $\mathcal{C}_f$ with $f(x)=\mathrm{Tr}_{q^n/q}(x)$.
This is a consequence of \cite[Theorem 3.7]{CsMP}.

\smallskip

\textbf{Examples of $(n-2)$-club rank metric codes.}\\
Up to equivalence, they are all of the form $\mathcal{C}_f$ with $f(x)=\mathrm{Tr}_{q^n/q}(\xi_{n-2}x)+b\mathrm{Tr}_{q^n/q}(\xi_{n-1}x)$, where $n=t(\ell+1)$ and there exist $a,b,c, \omega \in \F_{q^n}^*$, $\overline{S} \subseteq \F_{q^n}$ such that:
    \begin{itemize}
        \item $\F_q(b)=\F_{q^t}$, with $1 \leq t \leq n$,
        \item $\overline{S}$
        $\F_{q^t}$-subspace of dimension $\ell$ such that $c\F_{q^t} \cap \overline{S}=\{0\}$,
        \item $a \notin \overline{S}\oplus c\langle 1,b,\ldots,b^{t-2} \rangle_{\F_q} $,
        \item $(s_1,\ldots,s_{\ell t},c,cb,\ldots,cb^{t-3},1+\omega, a+\omega b)$ is an ordered $\fq$-basis and $(\xi_0,\ldots,\xi_{n-1})$ is its dual basis.
    \end{itemize} 
This follows by Theorem \ref{th:(n-2)clubpoldescript}.

\smallskip

\textbf{Examples of $i$-club rank metric codes.}\\
Let $n=rt$ with $t,r>1$.
Examples of $i$-club rank metric codes are the code $\mathcal{C}_f$ with $f(x)=\mathrm{Tr}_{q^n/q^t}(g(x)-x)$, where $g(x)\in \mathcal{L}_{t,q}$ is a scattered $q$-polynomial.
This follows by Theorem \ref{th:polscattclub}.

\subsection{$3$-dimensional linear rank metric codes with minimum distance $1$}

As a consequence on the classification of the linear R\'edei blocking sets of Theorem \ref{th:classRedeibs} we also obtain the following classification of rank metric codes with the following properties: let $\C$ be a linear $[n+1,3,1]_{q^n/q}$ non-degenerate rank metric code with at least a codeword of weight $n-1$.
Let $U$ be any system associated with $\C$ and let $W$ be a $2$-dimensional $\F_{q^n}$-subspace such that $\dim_{\fqn}(\langle U\cap W\rangle_{\fqn})=2$, then we say that $U$ is \emph{$q$-non-degenerate}.

\begin{theorem}
Let $\C$ be a linear $[n+1,3,1]_{q^n/q}$ $q$-non-degenerate rank metric code. Then any associated system with $\C$ is a linear R\'edei type blocking set in $\PG(2,q^n)$. Then
\begin{itemize}
    \item If $\C$ has more than $q^n-1$ codewords of weight one, then $\C$ is semilinearly equivalent to the code generated by 
\[ G=\left(
\begin{array}{ccccc}
    \xi_1 &   \xi_2 & \cdots & \xi_n & 0 \\
    \mathrm{Tr}_{q^n/q}(\xi_1) & \mathrm{Tr}_{q^n/q}(\xi_2) & \cdots & \mathrm{Tr}_{q^n/q}(\xi_n) & 0\\
    0 & 0 & \cdots & 0 & 1
\end{array}
\right),
\]
where $\xi_1,\ldots,\xi_n$ is an $\fq$-basis of $\fqn$.
\item If $n=4$, then $\C$ is semilinearly equivalent to the code generated by 
\[ G=\left(
\begin{array}{ccccc}
    \xi_1 &   \xi_2 & \cdots & \xi_n & 0 \\
    f(\xi_1) & f(\xi_2) & \cdots & f(\xi_n) & 0\\
    0 & 0 & \cdots & 0 & 1
\end{array}
\right),
\]
where $\xi_1,\ldots,\xi_n$ is an $\fq$-basis of $\fqn$ and $f$ is one of the polynomial described in \cite[Theorem A.3]{NPSZ2022complwei}.
\item If $\C$ has exactly $q^n-1$ codewords of weight one, $(q^n-1)q^2$ codewords of weight two and all the remaining nonzero codewords have weight greater than or equal to $n$, then $\C$ is semilinearly equivalent to the code generated by 
\[ G=\left(
\begin{array}{ccccc}
    \xi_1 &   \xi_2 & \cdots & \xi_n & 0 \\
    f(\xi_1) & f(\xi_2) & \cdots & f(\xi_n) & 0\\
    0 & 0 & \cdots & 0 & 1
\end{array}
\right),
\]
where $\xi_1,\ldots,\xi_n$ is an $\fq$-basis of $\fqn$ and $f$ is as described in Theorem \ref{th:(n-2)clubpoldescript}.
\end{itemize}
\end{theorem}
\begin{proof}
Let $U$ be any associated system with $\C$. Then $\dim_{\fq}(U)=n+1$ and, since the minimum distance is one, by Theorem \ref{th:connection} there exists at least one $\fqn$-subspace of dimension $2$ such that $\dim_{\fq}(U\cap W)=n$. Therefore, $L_U$ is a linear blocking set of R\'edei type with $\ell=\PG(W,\fqn)$ an its R\'edei line. 
Moreover, since $\C$ is $q$-non-degenerate, then there exists $\overline{W}$ $\fqn$-subspace of dimension $2$ such that $\dim_{\fq}(U\cap \overline{W})=2$ and $\dim_{\fqn}(\langle U\cap \overline{W}\rangle_{\fqn})=2$, that is $L_U$ has a $(q+1)$-secant line.
\begin{itemize}
    \item If $\C$ has more than $q^n-1$ codewords of weight one, then there is more than one R\'edei line to $L_U$. Moreover, since there exists a codeword of weight $n-1$. By \cite[Theorem 4.1]{DeBeuleVdV} we have that $L_U$ has size at least $q^n+q^{n-1}+1$.
    In \cite[Theorem 5]{LunPol2000}, it has been proved that an $\fq$-linear blocking set in $\PG(2,q^n)$ of R\'edei type having size at least $q^n+q^{n-1}+1$ and with at least two R\'edei lines is $\mathrm{PGL}(2,q^n)$-equivalent to $L_U'$, where
    \[ U'=U_{\mathrm{Tr}_{q^n/q}}\oplus \langle (0,0,1) \rangle_{\fq}. \]
    By Lemma \ref{lem:equivBoPol}, then $U$ and $U'$ are $\mathrm{GL}(3,q^n)$-equivalent and hence the assertion.
    \item This point follows from the classification result \cite[Section 4]{BoPol} and from Lemma \ref{lem:equivBoPol} and Theorem \ref{th:semilinearequivalence}.
    \item Because of the conditions on the weight distribution of $\C$, we have that $L_{U\cap W}$ is an $(n-2)$-club contained in the line $\PG(W,\fqn)$. Therefore, the result follows from Theorem \ref{th:(n-2)clubpoldescript} and again from Lemma \ref{lem:equivBoPol} and Theorem \ref{th:semilinearequivalence}.
\end{itemize}
\end{proof}

\section{Conclusions and open problems}

In this paper we have investigated clubs of rank $n$ in $\PG(1,q^n)$.
We have been able to provide a classification result for $(n-2)$-clubs. Then we have analyzed the $\GammaL(2,q^n)$-equivalence of the known subspaces defining clubs, for some of them the problem is then translated in determining whether or not certain scattered spaces are equivalent. Then we have detected all the linearized polynomials defining the known families of clubs. Then we have applied our results to the theory of blocking sets, KM-arcs and rank metric codes.

Here, we list some open problems and questions that naturally arise:

\begin{itemize}
    \item To classify $i$-clubs for $i<n-2$. 
    \item Can a coding-theoretical approach be used  to the first problem?
    \item Recently De Boeck and Van de Voorde in \cite{DeBoeckVdV20} proved that $2$-clubs do not exist in $\PG(1,q^5)$. Can this result be extended to $\PG(1,q^n)$ for any value of $n\geq 5$?
    \item As for the blocking sets, are KM-arcs constructed from two $\GammaL(2,q^n)$-inequivalent clubs  $\mathrm{P\Gamma L}(3,q^n)$-inequivalent?
    \item Can we find other conditions on the parameters of a linear rank metric code in such a way that results on linear blocking sets can be still used to classify them?
    \item Similarly to what happens to the traditional KM-arcs, can linear sets be used to construct example of generalized KM-arcs introduced in \cite{CsWei}?
\end{itemize}

\section*{Acknowledgement}

The authors were supported by the project ``VALERE: VAnviteLli pEr la RicErca" of the University of Campania ``Luigi Vanvitelli'' and by the Italian National Group for Algebraic and Geometric Structures and their Applications (GNSAGA - INdAM).

\end{document}